\documentclass[12pt,a4paper]{amsart}\usepackage[a4paper, left=28mm, right=28mm, top=28mm, bottom=34mm]{geometry}
\usepackage{tikz-cd}
\usepackage[hyphens]{url}
\usepackage{hyperref}
\makeatletter

\@addtoreset{equation}{section}
\makeatother

\usepackage{amsmath, amssymb}
\usepackage{latexsym}
\usepackage{amsthm}
\usepackage{amscd}
\usepackage{mathrsfs}
\usepackage[all]{xy}
\usepackage{graphicx}
\usepackage{bm}
\usepackage{comment}
\usepackage{xcolor}

\newcommand{\Q}{\mathbb{Q}}
\newtheorem*{theoremA}{Theorem A}
\newtheorem*{theoremB}{Theorem B}
\newtheorem*{theoremC}{Theorem C}
\newtheorem*{theoremD}{Theorem D}

\newtheorem{theorem}{Theorem}[section]
\newtheorem{example}[theorem]{Example}
\newtheorem{lemma}[theorem]{Lemma}
\newtheorem{corollary}[theorem]{Corollary}

\newtheorem{proposition}[theorem]{Proposition}
\newtheorem{remark}[theorem]{Remark}
\newtheorem{definition}[theorem]{Definition}

\DeclareMathOperator{\Fr}{Fr}
\DeclareMathOperator{\Gal}{Gal}

\DeclareMathOperator{\Ker}{Ker}
 
\DeclareMathOperator{\Tr}{Tr}

\DeclareMathOperator{\Spec}{Spec}

\newcommand{\F}{\mathbb{F}}

\begin{document}
\title[Gauss--Heilbronn Sums and Coverings of Deligne--Lusztig Type Curves]
{Gauss--Heilbronn Sums and Coverings of Deligne--Lusztig Type Curves}

\author{Tetsushi Ito}
\address{
Department of Mathematics, Faculty of Science, Kyoto University
Kyoto, 606--8502, Japan}
\email{tetsushi@math.kyoto-u.ac.jp}

\author{Daichi Takeuchi}
\address{Department of Mathematics,
Institute of Science Tokyo,
2-12-1 Ookayama, Meguro-ku, Tokyo, 152-8551, Japan
}
\email{daichi.takeuchi4@gmail.com}

\author{Takahiro Tsushima}
\address{
Keio University School of Medicine,
4-1-1 Hiyoshi, Kohoku-ku,
Yokohama, 223-8521, Japan}
\email{tsushima@keio.jp}

\date{\fbox{November ??, 2025}}


\date{}

\begin{abstract}
We study exponential sums on Witt vectors, known as Gauss--Heilbronn sums, and the curves whose Frobenius traces realize these sums via a Deligne--Lusztig type construction.
For $3$-typical Witt vectors of length two, we  analyze Gauss--Heilbronn sums, from which we fully determine the Frobenius slopes of the associated curves. 
%This investigation also yields an extension of Carlitz's classical formula for cubic exponential sums.
\end{abstract}

\maketitle
\vspace{-2em}

\section{Introduction}

Exponential sums over finite fields have been studied extensively for their 
intrinsic arithmetic properties and their connections to the geometry of 
algebraic curves. Classical examples include cubic and Kloosterman-type sums; 
see, for instance, \cite{B, Ca, Ca2, Da, Del, Ka0, Ka00, Liv}. 
Even in these basic yet nontrivial cases, such sums already exhibit rich arithmetic phenomena.

A natural generalization arises from the theory of Witt vectors, whose higher-level additive characters give rise to a richer family of exponential sums.
Let $q$ be a power of a prime $p$, let $l$ be a prime, and let 
$W_{n,q}$ denote the ring of $l$-typical Witt vectors of length $n$ over $\F_q$.
Associated to characters of $W_{n,q}$, one can define exponential sums that carry significant arithmetic content and reveal 
intricate connections between number-theoretic and geometric structures; see, for example, \cite{DWX, Ka0a, Ka1, Ka2, Ka3, K, KU1, KU2, Li, LW}. 
For an integer $m \ge 1$, which plays the role of a tame level when $p=l$, we introduce a Deligne--Lusztig type affine curve $C_{m,n,l,q}/\F_q$ 
whose Frobenius traces realize the exponential sums
\[
\mathcal S
 = \biggl\{
   -\sum_{x\in\F_q}
    \psi\!\left(\sum_{i=0}^{n-1} a_i x^{l^i m}\right)
   \,\biggm|\,
   (a_0,\dots,a_{n-1})\in\F_q^n\setminus\{0\}
  \biggr\},
\quad \text{in the case}\ p\neq l,
\]
and
\[
\mathcal T
 = \biggl\{
   -\sum_{x\in\F_q}
    \xi\bigl(a\cdot(x^m,0,\dots,0)\bigr)
   \,\biggm|\,
   a\in W_{n,q}\setminus\{0\}
  \biggr\},
\quad \text{in the case}\ p=l,
\]
where $\psi$ is a nontrivial additive character of $\F_q$ and 
$\xi$ is an additive character of $W_{n,q}$ of order $p^n$.
If $m=1$, the sums $\mathcal T$ are referred to as Gauss--Heilbronn sums in \cite{Li}.
In the following, we simply write $C_{m,n}$ for 
$C_{m,n,l,q}$. 

\medskip
In the case $n=2$, we give an 
explicit formula for the $L$-polynomials of
$\overline{C}_{1,2}$. 
Next, assuming $p=l$, we study the finite \'etale covering
\[
C_{P,n} \longrightarrow C_{1,n},
\]
which is associated with a separable $\F_p$-linearized polynomial $P(x) \in \F_q[x]$.  
For a smooth affine curve $X$, let $\overline{X}$ denote 
its smooth compactification. 
If $n=2$ and $p=2$, the curves $\{\overline{C}_{P,2}\}_P$ can be identified with the van der Geer--van der Vlugt curves \cite{ITT}.  
IN particular, we give an 
explicit formula for the $L$-polynomials of
$\overline{C}_{P,2}$ in the case $n=2$ and $p=3$.  
Consequently, the Frobenius slopes of $\overline{C}_{P,2}$ are $\{1/3,\, 2/3\}$. 
By varying the choice of $P$, we enlarge the class of curves whose Frobenius slopes can be completely determined.  
We also determine the 
Frobenius slopes of the smooth compactification $\overline{C}_{2,2}$. 

\medskip
Let $W=W_{2,q}=\F_q^2$ with abelian group law
\[
(a,b)+(a',b') = (a+a',\; b+b' - aa'(a+a')).
\]
Fix a prime $\ell\neq p$. 
Let $W^\vee_{\mathrm{prim}}$ denote 
the set of $\overline{\mathbb{Q}}_{\ell}$-valued characters whose restriction 
to $\{0\}\times\F_q$ is nontrivial.  
For a smooth projective curve $X/\F_q$, its 
$L$-polynomial is
\[
L_{X/\F_q}(T)
 := \det\!\left(1-\Fr_q^\ast T;\,H^1(X_{\F},\overline{\Q}_\ell)\right)
 \in\Q[T], 
\]
where $\F$ is an algebraic closure of $\F_q$. 
For $\xi \in W_{\mathrm{prim}}^{\vee}$, 
let 
\[
S_{\xi}:=\sum_{x \in \F_q} \xi(x,0). 
\]
As one of the principal results, we obtain an explicit factorization of the $L$-polynomial of $\overline{C}_{1,2}$.
The statement is as follows.
\begin{theoremA}[Explicit $L$-polynomial for $n=2$] 
Assume $n=2$ and $l=3$. 
Then we have 
\[
L_{\overline{C}_{1,2}/\F_q}(T)=\prod_{\xi \in W_{\rm prim}^{\vee}}
\left(1+S_{\xi} T+qT^2\right). 
\]
\end{theoremA}
This may be viewed as a generalization of Carlitz's formula \cite{Ca} for cubic exponential sums. 

We now focus on the case $p=3$ and study the $3$-adic valuation of $S_{\xi}$.
\begin{theoremB}[Valuation of 
Gauss--Heilbronn sums]\label{tA}
Assume $n=2$ and $p=l=3$. 
Let $\zeta_9:=e^{2\pi i/9}$ and let 
$\mathfrak{p}$ denote a prime of 
$\Q(\zeta_9)$ 
lying over $(3)$. 
Let $v_{\mathfrak{p}}(-)$ denote the 
valuation of $\Q(\zeta_9)$ at $\mathfrak{p}$.  
We write $q=3^f$. 
Then we have 
\[
v_{\mathfrak{p}}(S_{\xi})=\frac{f}{3}. 
\]
\end{theoremB}
This theorem is proved by comparing the exponential sums with the Frobenius traces of certain supersingular curves.

In \cite[Theorem~1.3]{Li}, Liu computes the Frobenius slopes of $\overline{C}_{1,n}$.
While Liu states the result for general $q$, the proof crucially relies on the assumption $q=p$ (see Remark~\ref{Cliu} for details). 
Our argument applies uniformly for all $q=3^f$, including the case $f>1$, where Liu's method seems not to apply.
Thus Theorem~A provides a genuinely new valuation result for Gauss--Heilbronn sums.

As an application of Theorem~B, we deduce the following result for $m=2$.

\begin{theoremC}[Frobenius slopes]
Assume $m=2$, $n=2$ and $p=l=3$. 
The Frobenius slopes of $\overline{C}_{2,2}$ with respect to
$\mathrm{Fr}_q^\ast$ are 
\[
\left\{\frac{1}{2},\  \frac{1}{3},\ \frac{2}{3},\ 
\frac{1}{6},\ \frac{5}{6}\right\}. 
\]
\end{theoremC}

Using Theorem~B, we obtain
\[
v_{\mathfrak{p}}\biggl(\sum_{x \in \F_q} \nu_q(x)\,
\xi(x,0)\biggr)=\frac{f}{6},
\]
where $\nu_q(x):=x^{\frac{q-1}{2}}$ is the quadratic character
and $\xi \in W_{\mathrm{prim}}^{\vee}$.  
Combining this valuation with the product formula in \cite{La} 
yields Theorem~C.

\medskip

For a finite extension $\F_{p^a}/\F_{p^b}$, let $\Tr_{p^a/p^b} \colon 
 W_{n, p^a} \to W_{n, p^b}$ denote the trace map.

 We define the curve $C_{P,n}$ associated with a separable $\F_p$-linearized polynomial $P(x) \in \F_q[x]$.
The curve $C_{P,n}$ is geometrically connected. 
We will study the $L$-polynomial and Frobenius slopes of the projective curve $\overline{C}_{P,n}$ in the case $n=2$ and $p=l=3$.

Using Theorems~A and B, we show the following. 
\begin{theoremD}[Explicit $L$-polynomial for $n=2$, $p=l=3$]
Assume $n=2$ and $p=l=3$.  We take a nontrivial additive 
character $\psi$ of $\F_q$.
For a suitable extension $\F_{q^s}$ one has
\begin{equation}\label{LL}
L_{\overline{C}_{P,2}/\F_{q^s}}(T)
= \prod_{\xi\in W^\vee_{\mathrm{prim}}}
  \prod_{\alpha\in K_{P,\psi}}
  \biggl(
    1 
    + \sum_{x\in\F_{q^s}}
        \psi(\Tr_{q^s/q}(\alpha x))\,
        \xi(\Tr_{q^s/q}(x,0))\,T
    + q^s T^2
  \biggr),
\end{equation}
where $K_{P,\psi}\subset \F_{q^s}$ is a finite explicitly determined subset depending on $P$ and $\psi$.
As a consequence, the Frobenius slopes of $\overline{C}_{P,2}$ over $\F_q$ 
are precisely $\{1/3,2/3\}$.
\end{theoremD}
The proof proceeds by reducing the problem to Theorem~A and then applying 
Theorem~B to evaluate the resulting exponential sums.

\medskip

\subsubsection*{Organization of the paper.}
Section~\ref{s2} introduces the Deligne--Lusztig type curve $C_{m,n}$ and develops its basic properties for general levels. 
Section~\ref{s3} proves Theorem~A and clarifies the supersingularity phenomena for $n=2$ and $l=3$, showing that $\overline{C}_{m,n}$ is supersingular when $p=2$ or $5$. 
In contrast, when $p=l$ (with $l$ an arbitrary prime) or $p\equiv1\pmod{l}$, we show that $\overline{C}_{m,n}$ is not supersingular. 
The proofs of these cases rely on Katz's monodromy theory and Stickelberger's theorem, respectively, and are independent of our main valuation results. 
Section~\ref{s3.5} treats the case $p=3$, and proves Theorems~B and C. 
Section~\ref{s4} studies the coverings $C_{P,n}$ of $C_{1,n}$ and establishes Theorem~D.

\section{Deligne--Lusztig type construction}\label{s2}
In this section, we introduce the curve $C_{m,n,l,q}$, which is the main object of study in this paper. 

Let $l$ be a fixed prime number and $n \ge 1$. 
The $l$-typical Witt vector scheme of length $n$, denoted by $\mathbb{W}_n$, is defined over $\mathbb{Z}$ as the affine scheme
\[
\mathbb{W}_n := \Spec \mathbb{Z}[X_0, X_1, \dots, X_{n-1}],
\]
where the coordinates $X_i$ correspond to the components of a Witt vector.

To define the addition on $\mathbb{W}_n$, we use the ghost map
\[
\operatorname{gh} \colon \mathbb{W}_n \longrightarrow \mathbb{G}^n_{\mathrm{a},\,\mathbb{Z}}, \quad 
X \mapsto (w_0(X), \dots, w_{n-1}(X)),
\]
where the ghost coordinates $w_i$ are integer-coefficient polynomials in $X_0, \dots, X_{n-1}$ given by
\[
\begin{aligned}
w_0(X) &= X_0, \\
w_1(X) &= X_0^l + l X_1, \\
w_2(X) &= X_0^{l^2} + l X_1^l + l^2 X_2, \\
&\;\;\vdots
\end{aligned}
\]
The addition and multiplication on $\mathbb{W}_n$ are defined uniquely by the conditions
\[
\operatorname{gh}(X + Y) = \operatorname{gh}(X) + \operatorname{gh}(Y),
\qquad
\operatorname{gh}(X \cdot Y) = \operatorname{gh}(X) \cdot \operatorname{gh}(Y),
\]
so that both operations become coordinate-wise in the ghost components. 
In other words, the Witt vector sum and product
\begin{equation}\label{S}
\begin{aligned}
X + Y &:= (S_0(X,Y), \dots, S_{n-1}(X,Y)), \\ 
X \cdot Y &:= (P_0(X,Y), \dots, P_{n-1}(X,Y))
\end{aligned}
\end{equation}
are determined by integer polynomials $S_i(X,Y), P_i(X,Y)$ satisfying the above identities
(cf.\ \cite[II\S6]{Se}).

Let $q$ be a power of a prime number $p$. 
We let 
\begin{equation}\label{wnq}
\mathbb{W}_{n,q}:=\mathbb{W}_n \otimes_{\mathbb{Z}} \F_q, \qquad 
W_{n,q}:=\mathbb{W}_{n,q}(\F_q). 
\end{equation}
Let 
\[
L_q \colon \mathbb{W}_{n,q}
\to \mathbb{W}_{n,q}, \quad (x_0,\ldots,x_{n-1})
\mapsto (x_0^q,\ldots,x_{n-1}^q)-
(x_0,\ldots,x_{n-1})
\]
be the Lang torsor. 
\begin{definition}
Let $m$ be a positive integer prime to $p$. 
Let 
\[
i \colon \mathbb{A}_{\F_q}^1 \hookrightarrow \mathbb{W}_{n,q}, \quad 
t \mapsto (t^m, 0,\ldots,0)
\]
be the closed immersion. 
Let $C_{m,n,l,q}$ denote the 
affine curve over $\F_q$ defined by 
the cartesian diagram
\[
\xymatrix{
C_{m,n,l,q} \ar[r]\ar[d]^{\pi} & \mathbb{W}_{n,q} \ar[d]^{L_q}\\
\mathbb{A}_{\F_q}^1 \ar[r]^-i & \mathbb{W}_{n,q}. 
}
\] 
\end{definition}
Then $C_{m,n,l,q}$ admits a natural action of 
$W_{n,q}:=\mathbb{W}_{n,q}(\F_q)$. 
The morphism 
\begin{equation}\label{pi}
\pi \colon C_{m,n,l,q} \to \mathbb{A}_{\F_q}^1, \quad 
(t, x_0,\ldots,x_{n-1}) \mapsto t, 
\end{equation}
is a finite \'etale Galois covering whose Galois group 
is $W_{n,q}$. 
If $m=1$, the parameter $t$ is 
expressed as $x_0^q-x_0$.

For $\xi \in W_{n,q}^{\vee}$, 
let $\mathscr{L}_{\xi}(x_0,\ldots,x_{n-1})$ denote the 
sheaf of rank one 
on $\mathbb{W}_{n,q}$ 
defined by the $W_{n,q}$-covering 
$L_q \colon \mathbb{W}_{n,q} \to
\mathbb{W}_{n,q}$ and $\xi$. 
For a morphism of $\F_q$-schemes 
$(f_0,\ldots, f_{n-1}) \colon X \to 
\mathbb{W}_{n,q},\ x \mapsto 
(f_0(x), \ldots, f_{n-1}(x))$, 
let $\mathscr{L}_{\xi}(f_0, \ldots, f_{n-1})$
denote the pull-back 
\[
(f_0,\ldots, f_{n-1})^\ast \mathscr{L}_{\xi}(x_0,\ldots,x_{n-1}).
\]
We simply write $\mathscr{L}_{\xi}(t^m,0'\mathrm{s})$ for 
$\mathscr{L}_{\xi}(t^m,0,\ldots,0)$.

Let $\F$ be an algebraic closure of $\F_q$. 
We write $\mathbb{A}^1$
for the affine line over $\F$. 
We take a prime number $\ell \neq p$.
For an algebraic variety $X$ over a field 
$F$ and a field extension $E/F$, let 
$X_E$ denote the base change of $X$ to $E$. 
For an algebraic variety $X$ over $\F_q$, 
we denote by $H_{\rm c}^i(X)$ and 
$H^i(X)$ for 
$H_{\rm c}^i(X_{\mathbb{F}},\overline{\mathbb{Q}}_{\ell})$ and $H^i(X_{\mathbb{F}},\overline{\mathbb{Q}}_{\ell})$, respectively. 
For a finite abelian group $A$, 
let $A^{\vee}$ denote the character group 
$\mathrm{Hom}_{\mathbb{Z}}(A,\overline{\mathbb{Q}}_{\ell}^{\times})$. 

\begin{lemma}\label{ll}
We have an isomorphism 
\[
H^1_{\rm c}(C_{m,n,l,q}) \simeq 
\bigoplus_{\xi \in W_{n,q}^{\vee}} 
H_{\rm c}^1(\mathbb{A}^1,\mathscr{L}_{\xi}(t^m,0'\mathrm{s})). 
\]
\end{lemma}
\begin{proof}
Let $\pi$ be as in \eqref{pi}. 
We have $\pi_{\ast} \overline{\mathbb{Q}}_{\ell}
\simeq \bigoplus_{\xi \in W_{n,q}^{\vee}} \mathscr{L}_{\xi}(t^m,0'\mathrm{s})$. 
Taking $H_{\rm c}^1(\mathbb{A}^1,-)$ gives the claim. 
\end{proof}
\begin{lemma}\label{gt}
Assume $p \neq l$. 
The ghost morphism 
\[
\mathrm{gh} \colon \mathbb{W}_{n,q}
\to \mathbb{G}_{\mathrm{a},\,\F_q}^n, \quad 
x=(x_0,\ldots,x_{n-1}) \mapsto (w_0(x),\ldots,w_{n-1}(x))
\]
is an isomorphism of ring schemes.
\end{lemma}
\begin{proof}
This is well-known (cf.\ \cite[II\S6, Th\'eor\`eme 6]{Se}). 
\end{proof}
\begin{lemma}\label{defining}
Assume $p \neq l$. The curve $C_{m,n,l,q}$ is 
defined by 
\begin{align*}
 w_0^q-w_0&=t^m, \\
 w_1^q-w_1&=t^{lm}, \\
&  \vdots \\
 w_{n-1}^q-w_{n-1}&=t^{l^{n-1}m}. 
\end{align*}
\end{lemma}
\begin{proof}
We consider the defining equation of 
$C_{m,n,l,q}$: 
\begin{equation}\label{gtt}
(x_0^q,\ldots,x_{n-1}^q)=(x_0,\ldots,x_{n-1})+(t^m,0\ldots,0). 
\end{equation}
We note $w_i(x_0^q,\ldots,x_{n-1}^q)=w_i(x_0,\ldots,x_{n-1})^q$. 
We have 
\begin{equation}\label{w}
w_i(t^m,0'\mathrm{s}):=w_i(t^m,0,\ldots,0)=t^{l^i m} \qquad \textrm{for $0 \leq i \leq n-1$}. 
\end{equation}
Let $x=(x_0,\ldots,x_{n-1})$. 
By the definition of the sum of 
$\mathbb{W}_{n,q}$ and Lemma \ref{gt}, the equality \eqref{gtt} is equivalent to 
\begin{align*}
 w_0(x)^q&=w_0(x)+t^m, \\
 w_1(x)^q&=w_1(x)+t^{lm}, \\
& \  \vdots \\
 w_{n-1}(x)^q&=w_{n-1}(x)+t^{l^{n-1}m}. 
\end{align*}
\end{proof}
\begin{lemma}\label{pneql2}
Assume $p \neq l$. 
We identify $W_{n,q}=\mathbb{W}_{n,q}(\F_q)$
with $\F_q^n$ via $\mathrm{gh}$. 
For $a=(a_0,\ldots,a_{n-1}) \in \F_q^n$, 
let $f_{a}(t):=\sum_{i=0}^{n-1} a_i t^{l^i m}$. 
Let $\psi \in \F_q^{\vee} \setminus \{1\}$. 
Let $\xi \in W_{n,q}^{\vee}$, which corresponds to 
the character of $\F_q^n$ defined by $x:=(x_0,\ldots,x_{n-1}) \mapsto 
\psi(\sum_{i=0}^{n-1}a_i x_i)$ with 
$a=(a_0,\ldots,a_{n-1}) \in \F_q^n$. 
Then 
\[
\mathscr{L}_{\xi}(t^m,0'\mathrm{s})
\simeq \mathscr{L}_{\psi}(f_{a}(t)). 
\]
We also have 
\[
\sum_{t \in \F_q} \xi(t^m,0'\mathrm{s})
=\sum_{t\in \F_q} \psi(f_{a}(t)). 
\]
\end{lemma}
\begin{proof}
By Lemma \ref{gt} and \eqref{w}, we have 
\begin{equation}\label{pn}
\mathscr{L}_{\xi}(t^m,0'\mathrm{s})
\simeq \mathscr{L}_{\psi}\biggl(\sum_{i=0}^{n-1}a_i w_i(t^m,0'\mathrm{s})\biggr)  
\simeq \mathscr{L}_{\psi}(f_{a}(t)).
\end{equation}
Finally, the map $\mathrm{gh}\colon W_{n,q} \xrightarrow{\sim} \F_q^n$
is an isomorphism, and applying the same identity
$w_i(t^m,0'\mathrm{s}) = t^{l^i m}$
gives
\[
\sum_{t\in \F_q} \xi(t^m,0'\mathrm{s})
   = \sum_{t\in \F_q} \psi(f_a(t)).
\]
\end{proof}
\begin{corollary}\label{pneql}
Assume $p \neq l$. We take a 
nontrivial character $\psi \in \F_q^{\vee}
\setminus \{1\}$. 
For $a=(a_0,\ldots,a_{n-1}) \in \F_q^n$, 
let $f_{a}(t):=\sum_{i=0}^{n-1} a_i t^{l^i m}$. 
Then we have an isomorphism 
\[
H_{\rm c}^1(C_{m,n,l,q}) \simeq 
\bigoplus_{a \in \F_q^n}
H_{\rm c}^1(\mathbb{A}^1,\mathscr{L}_{\psi}
\left(f_{a}(t)\right)). 
\]
Furthermore, we have 
\begin{align*}
\dim 
H_{\rm c}^1(\mathbb{A}^1,\mathscr{L}_{\psi}
\left(f_{a}(t)\right))&=\deg (f_{a}(t)-1) \quad \textrm{for $a \neq 0$}, \\ 
\dim H_{\rm  c}^1(C_{m,n,l,q})&=
\sum_{a \in \F_q^n \setminus \{0\}}
\deg (f_{a}(t)-1)=(q-1)\sum_{i=0}^{n-1} q^i(m l^i-1). 
\end{align*}
\end{corollary}
\begin{proof}
The first claim follows from Lemmas \ref{ll}
and \ref{pneql2}. 

The latter claim follows from $H_{\rm c}^i(\mathbb{A}^1,\mathscr{L}_{\psi}
\left(f_{a}(t)\right))=0$ for $i \neq 1$ and 
the Grothendieck--Ogg--Shafarevich formula, since the Swan conductor exponent of $\mathscr{L}_{\psi}
\left(f_{a}(t)\right)$ is $\deg f_{a}(t)$ by 
$p \nmid l m$. 
The last equality is directly checked. 
\end{proof}
Now, we treat the case $p=l$. 
We now recall a result of \cite{Br} and \cite[Lemma 3.1]{Ka0a}. 
Let   
$W_{n,p}:=\mathbb{W}_n(\F_p)$ and 
let $\Tr_{q/p} \colon W_{n,q} \to W_{n,p}$
denote the trace map. 
Let $W_{n,q,\mathrm{prim}}$ denote the 
set of characters $\xi$ of $W_{n,q}$ such that
$\xi|_{\{0'\mathrm{s}\} \times \F_q} \neq 1$.  
\begin{lemma}\label{ch}
Let ${}^0 \xi \in W_{n,p}^{\vee}$ be a faithful character. 
For $a \in W_{n,q}$, define $\xi_a \in W_{n,q}^{\vee}$
by 
\[
\xi_{a}(x)=
{}^0 \xi\circ \Tr_{q/p}(ax). 
\]
Then the map 
\[
W_{n,q} \longrightarrow 
W_{n,q}^{\vee}, \quad 
a \longmapsto \xi_a
\]
is a bijection. 
For $a=(a_0,\ldots,a_{n-1}) \in W_{n,q}$, we have 
\[
\xi_a \in W_{n,q, \mathrm{prim}}^{\vee}
\iff a_0 \neq 0. 
\]
\end{lemma}
\begin{proof}
The former claim is well-known. 

The latter claim follows from 
$\xi_{a}(0'\mathrm{s}, x)
={}^0 \xi\bigl(\Tr_{q/p}(0'\mathrm{s},a_0^{p^{n-1}}x)\bigr)$ for $x \in \F_q$. 
\end{proof}

\begin{lemma}[{\cite[Lemma 3.1]{Ka0a}}]\label{lemma:Brylinski}
If $a = (a_0, \dots, a_{n-1}) \in W_{n,q}$ is nonzero, then $\xi_a$ has order $p^{n-d}$, 
where $d$ is the largest integer such that $a_i = 0$ for all $i \le d-1$. 
For such an $a$, the Swan conductor exponent of 
$\mathscr{L}_{\xi_a}(t^m,0'\mathrm{s})$ at $\infty$ is given by
\[
\operatorname{Swan}_\infty
\big(\mathscr{L}_{\xi_a}(t^m,0'\mathrm{s})\big)
= m p^{n-1-d}.
\]
We have $\dim H_{\rm c}^1(\mathbb{A}^1,\mathscr{L}_{\xi_a}(t^m,0'\mathrm{s}))=m p^{n-1-d}-1$. In particular, if $\xi_a \in W_{n,q,\mathrm{prim}}^{\vee}$, 
we have $\dim H_{\rm c}^1(\mathbb{A}^1,\mathscr{L}_{\xi_a}(t^m,0'\mathrm{s}))=m p^{n-1}-1$. 
\end{lemma}
\begin{proof}
The first claim is just \cite[Lemma 3.1]{Ka0a}. 
The second claim follows from 
the Grothendieck--Ogg--Shafarevich formula
and $H_{\rm c}^i(\mathbb{A}^1,\mathscr{L}_{\xi_a}(t^m,0'\mathrm{s}))=0$ for $i \neq 1$. The last claim follows from Lemma \ref{ch}. 
\end{proof}
\begin{lemma}\label{cyclic}
Assume $p =l$ and write $q=p^f$. 
We have an isomorphism $W_{n,q} \simeq (\mathbb{Z}/p^n \mathbb{Z})^f$ as groups. 
\end{lemma}
\begin{proof}
We take a basis $\{v_1,\ldots, v_f\}$ of $\F_q$
over $\F_p$. 
Recall 
\begin{align*}
F &\colon \mathbb{W}_{n,q} \to \mathbb{W}_{n,q}, \quad (x_0,\ldots,x_{n-1}) \mapsto (x_0^p,\ldots,x_{n-1}^p), \\
V &\colon \mathbb{W}_{n,q} \to \mathbb{W}_{n,q}, \quad (x_0,\ldots,x_{n-1}) \mapsto (0,x_0,\ldots,x_{n-2}).
\end{align*}
By $FV=VF=p$,  
\[
p^j(v_i,0'\mathrm{s})=
(0, \dots, 0, v_i^{p^j}, 0, \dots, 0), 
\]
where $v_i^{p^j}$ appears in the $(j+1)$-th component. Thus $(v_i,0'\mathrm{s})$ has order $p^n$. 
Let 
\[
\varphi \colon (\mathbb{Z}/p^n \mathbb{Z})^f \to W_{n,q}, \quad 
(a_1,\ldots,a_f) \mapsto \sum_{i=1}^f a_i (v_i,0'\mathrm{s}). 
\]
Assume $\varphi(a_1,\ldots,a_f)=0$. 
Since the Witt addition coincides with coordinatewise addition on the first 
component, the first coordinate of $\varphi(a_1,\ldots,a_f)$ is 
\[
\sum_{i=1}^f (a_i \bmod p)\, v_i \in \F_q .
\]
Since $\{v_1,\ldots,v_f\}$ is an $\F_p$-basis, 
we obtain $a_i \equiv 0 \pmod{p}$ for all $i$.  By applying the same argument 
to the higher $p$-power components, we obtain 
$a_i \equiv 0 \pmod{p^n}$ for every $i$. 
Thus $\varphi$ is injective.  
Since the both sides of $\varphi$ have the same cardinality, we obtain the claim. 
\end{proof}
For a smooth projective curve $X$, let $g(X)$ denote its genus. 
\begin{corollary}\label{genus}
We have 
\[
g(\overline{C}_{m,n,l,q})
=\frac{q-1}{2}\sum_{i=0}^{n-1}q^i(m l^i-1).  
\]
\end{corollary}
\begin{proof}
We show that 
the canonical map
\[
H_{\rm c}^1(C_{m,n,l,q}) \to H^1(C_{m,n,l,q})
\]
is an isomorphism. 

Assume $p \neq l$. 
Then the forgetful map 
$H_{\rm c}^1(\mathbb{A}^1,\mathscr{L}_{\psi}(f_{a}(t))) \to 
H^1(\mathbb{A}^1,\mathscr{L}_{\psi}(f_{a}(t)))$
is an isomorphism if $a \neq 0$. 
Hence the claim follows from Lemma \ref{pneql}. 

Assume $p=l$. 
For $a \in W_{n,q} \setminus \{0\}$, 
the canonical map
\[
H_{\rm c}^1(\mathbb{A}^1,\mathscr{L}_{\xi_a}(t^m,0'\mathrm{s})) \to 
H^1(\mathbb{A}^1,\mathscr{L}_{\xi_a}(t^m,0'\mathrm{s})) 
\]
is an isomorphism since $\mathscr{L}_{\xi_a}(t^m,0'\mathrm{s})$ is ramified at $\infty$. 
Hence the claim follows from 
Lemma \ref{lemma:Brylinski}. 

By the claim, the natural map 
$H_{\rm c}^1(C_{m,n,l,q}) \to H^1(\overline{C}_{m,n,l,q})$ is an isomorphism. 
The dimension of  $H_{\rm c}^1(C_{m,n,l,q})$ is computed using Lemmas \ref{pneql} and \ref{lemma:Brylinski}.
\end{proof}
\begin{lemma}\label{trace}
\begin{itemize}
\item[{\rm (1)}]
We have 
\[
\Tr(\Fr_q^\ast; H_{\rm c}^1(\mathbb{A}^1,\mathscr{L}_{\xi}(t^m,0'\mathrm{s})))=-\sum_{t \in \F_q}\xi(t^m,0'\mathrm{s}). 
\]
\item[{\rm (2)}]
We take an isomorphism $\overline{\mathbb{Q}}_{\ell} \simeq \mathbb{C}$. 
Then we have 
\[
\biggl|\sum_{t \in \F_q}\xi(t^m,0'\mathrm{s})\biggr|
\leq \dim H_{\rm c}^1(\mathbb{A}^1,\mathscr{L}_{\xi}(t^m,0'\mathrm{s}))\, \sqrt{q} . 
\]
\end{itemize}
\end{lemma}
\begin{proof}
The claim (1) follows from the Grothendieck 
trace formula. 

The claim (2) follows from the Weil conjectures 
for curves. 
\end{proof}

 \begin{lemma}\label{ff}
 Let $1 \le n' \le n$. We write $q=p^f$. 
There exist finite morphisms 
\begin{align*}
& \phi_1 \colon C_{m,n,l,q} \to C_{1,n,l,q}, \quad 
(t,x_0,\ldots,x_{n-1}) \mapsto (x_0,\ldots,x_{n-1}), \\
& \phi_2 \colon C_{1,n,l,q} \to C_{1,n,l,p}, \quad 
(x_0,\ldots,x_{n-1}) \mapsto 
\sum_{i=0}^{f-1} (x_0^{p^i},\ldots,x_{n-1}^{p^i}), \\
& \phi_3 \colon 
C_{1,n,l,p} \to C_{1,n',l,p}, \quad 
(x_0,\ldots,x_{n-1}) \mapsto (x_0,\ldots,x_{n'-1}), 
\end{align*} 
where the sum in $\phi_2$ is taken in the commutative algebraic group  $\mathbb{W}_{n,q}$. 
 \end{lemma}
 \begin{proof}
The second claim follows from, for 
$(x_0,\ldots,x_{n-1}) \in \mathbb{W}_n$, 
\[
L_p(\phi_2(x_0,\ldots,x_{n-1}))=(x_0^q,\ldots,x_{n-1}^q)-(x_0,\ldots,x_{n-1})
=L_q(x_0,\ldots,x_{n-1}).
\] 
The first and third claims follow from definition. 
 \end{proof}

\section{Witt sums of length two}\label{s3}
\subsection{Length two}
In this section, we specialize the results of the previous section to the case 
$m=1$ and $n=2$, and discuss them in more detail
for use in the next section.

We define 
\[
g_l(x,y):=\frac{x^l+y^l-(x+y)^l}{l} \in \mathbb{Z}[x,y], 
\]
which we regard as an element of $\F_q[x,y]$. 
This is nothing but $S_1(x,y)$ in \eqref{S}.

The addition law of the 
algebraic group $\mathbb{W}_{2,q}$
is given by 
\[
(x,y)+(x',y')=(x+x', y+y'+g_l(x,x')). 
\]
If $l \neq 2$, we have $g_l(x,-x)=0$. 
Thus
\begin{align*}
L_q(x,y)=(x^q,y^q)-(x,y)&=
\begin{cases}
(x^q,y^q)+(-x,-y+x^2) & \textrm{if $l=2$}, \\
(x^q,y^q)+(-x,-y) & \textrm{if $l \neq 2$} \\
\end{cases} \\
&=
\begin{cases}
(x^q-x, y^q-y-x^{q+1}+x^2) & \textrm{if $l=2$}, \\
(x^q-x,y^q-y+g_l(x^q,-x)) & \textrm{if $l \neq 2$}. 
\end{cases}
\end{align*} 
Recall that $C_{1,2, l,q}$ is defined by the cartesian diagram
\[
\xymatrix{
C_{1,2, l,q} \ar[r]\ar[d]^{\pi} & \mathbb{W}_{2,q} \ar[d]^{L_q}\\
\mathbb{A}^1_{\mathbb{F}_q} \ar[r]^-i & 
\mathbb{W}_{2,q}, 
}
\]
where $i \colon \mathbb{A}_{\F_q}^1 \hookrightarrow \mathbb{W}_{2,q}$ is the closed immersion $t \mapsto (t,0)$. 
Then $C_{1,2, l,q}$ is a smooth affine curve defined by 
\[
\begin{cases}
y^q-y=-x(x^q-x) & \textrm{if $l=2$}, \\
y^q-y=-g_l(x^q,-x) & \textrm{if $l \neq 2$}.  
\end{cases}
\]
We simply write $C$ for $C_{1,2, l,q}$.  
The curve $C$ admits a natural action 
of $W:=\mathbb{W}_{2,q}(\F_q)$ as follows: 
\begin{equation}\label{xycd}
(x,y) \cdot (a,b)=(x+a, y+b+g_l(x,a)) \quad 
\textrm{for $(x,y) \in C, \ (a,b) \in W$}.
\end{equation}
%since $L_q((x,y) \cdot (a,b))=L_q(x,y)+L_q(a,b)=L_q(x,y)$. 
We recall that the morphism
\[
\pi \colon C \to \mathbb{A}_{\F_q}^1,\quad (x,y) \mapsto x^q - x,
\]
is a finite \'etale Galois covering whose Galois group is $W$ (cf.\ \eqref{pi}).
\begin{lemma}
The curve $\overline{C}_{1,2,2,q}$ is supersingular. 
\end{lemma}
\begin{proof}
The defining equation of 
$C_{1,2,2,q}$ is $y^q-y=x R(x)$, where 
$R(x):=-x^q+x$ is $\F_q$-linearized. 
Hence the claim follows from \cite{ITT, TT}.
\end{proof}
\begin{lemma}\label{de}
Assume $p \neq l$. Then $C$ is defined by 
$y^q-y=(x^q-x)^l$. 
\end{lemma}
\begin{proof}
By Lemma \ref{defining} for $m=1$ and $n=2$, the defining equation of 
$C$ is $w_0^q-w_0=t$ and $w_1^q-w_1=t^l$. 
Setting $x = w_0$ and $y = w_1$, we obtain
$y^q - y = (x^q - x)^l$,
as claimed.
\end{proof}
Let $\xi \in W_{\mathrm{prim}}^{\vee}$. 
For simplicity, we set  
\[
\mathscr{L}_{\xi}:=\mathscr{L}_{\xi}(x,0).
\] 
\begin{lemma}\label{dim}
We have $H_{\rm c}^1(\mathbb{A}^1,\mathscr{L}_{\xi})=l-1$. 
\end{lemma}
\begin{proof}
The claim follows from Lemmas \ref{pneql2}, \ref{pneql}
and \ref{lemma:Brylinski}. 
\end{proof}
Recall that $W_{\rm prim}^{\vee}$ is the set of 
characters of $W$ whose restriction to 
the subgroup 
$\{0\} \times \F_q$ is nontrivial. 

\begin{lemma}
\label{Fum}
We have an isomorphism 
\[
H_{\rm c}^1(C) \simeq \bigoplus_{\xi \in W_{\rm prim}^{\vee}} H_{\rm c}^1(\mathbb{A}^1,\mathscr{L}_{\xi}). 
\]
Further, we have 
$\dim H_{\rm c}^1(C)=(l-1)q(q-1)$. 
\end{lemma}
\begin{proof}
By Lemma \ref{ll}, 
\[
H_{\rm c}^1(C) \simeq \bigoplus_{\xi \in W^{\vee}} H_{\rm c}^1(\mathbb{A}^1,\mathscr{L}_{\xi}).
\]
The $(\{0\} \times \F_q)$-fixed part of 
$H_{\rm c}^1(C)$ is zero because $C/(\{0\} \times 
\F_q)$ is isomorphic to $\mathbb{A}^1_{\F_q}$. Hence the former claim follows. 

The latter claim follows from $\# W^{\vee}_{\rm prim}=q(q-1)$ and Lemma \ref{dim} (cf.\  Corollary \ref{genus}). 
\end{proof}
\begin{comment}
\begin{lemma}
We have 
\[
|C^{\Fr_q \circ (a,b)}|=
\begin{cases}
q^2 & \textrm{if $b= 0$}, \\ 
0 & \textrm{otherwise}. 
\end{cases}
\]
\end{lemma}
\begin{proof}
Let $(a,b) \in W$. 
We consider 
\begin{align*}
& x^q=x-a, \\
& y^q+g_l(x^q,a)+b=y, \\
& y^q-y=-g_l(x^q,-x).
\end{align*}
Substituting the first equality to the second one, 
we obtain 
\[
y^q-y+g_l(x-a,a)=-b. 
\]
On the other hand, $x^q-x=-a$ implies that 
\[
y^q-y=-g_l(x^q,-x)=-g_l(x-a,-x)=-g_l(x-a,a). 
\]
Hence if $b \neq 0$, we have 
$|C^{\Fr_q \circ (a,b)}|=0$. 
If $b=0$, we have $|C^{\Fr_q \circ (a,b)}|=q^2$. 
\end{proof}
\end{comment}
We define 
\[
S_{\xi}:=\sum_{x \in \F_q} \xi(x,0). 
\]
By Lemmas \ref{trace}(2) 
and \ref{dim}, 
\begin{equation}\label{abs}
|S_{\xi}| \leq (l-1) \sqrt{q}. 
\end{equation}
If $p \neq l$, by Lemma \ref{gt}, 
\begin{equation}\label{cca}
W \xrightarrow{\sim} \F_q^2, \quad 
(x_0,x_1) \mapsto (x_0,x_0^l+l x_1). 
\end{equation}
For $a \in \F_q$ and 
$\psi \in \F_q^{\vee}$, let $\psi_a \in \F_q^{\vee}$
denote the character $x \mapsto \psi(ax)$.  
\begin{lemma}\label{sum}
Assume $p \neq l$. 
Let $\xi \in W^{\vee}_{\rm prim}$.
Assume that via \eqref{cca},  
the character $\xi$ corresponds to 
$\psi_a \boxtimes \psi$ with 
$\psi \in \F_q^{\vee} \setminus \{1\}$
and $a \in \F_q$. 
We have 
\[
S_{\xi}=\sum_{x \in \F_q}
\psi (x^l+ax). 
\]
\end{lemma}
We write $q=p^f$. 
We define an equivalence relation on $\F_q^{\times} \times \F_q$ by 
\[
(a,b) \sim (a',b') \iff \textrm{there exists 
$c \in \F_p^{\times}$ such that 
$a'=ca$ and $b'=cb$}. 
\]
\begin{lemma}\label{coho}
For $(a,b) \in \F_q^{\times} \times \F_q$, 
let $C_{a,b}$ denote the affine curve 
over $\F_q$ defined by $y^p-y=a x^l+b x$. 
Assume $p \neq l$. 
Then we have 
\[
H^1(\overline{C}) \simeq 
\bigoplus_{[(a,b)] \in (\F_q^{\times} \times \F_q)/\sim}H^1
(\overline{C}_{a,b}). 
\]
\end{lemma}
\begin{proof}
We take ${}^0 \psi \in \F_p^{\vee} \setminus \{1\}$
and set $\psi={}^0 \psi \circ \Tr_{q/p}$. 
From $p \neq l$ and Lemma \ref{pneql}, it follows that  
\begin{align*}
H^1(\overline{C})&\simeq H_{\rm c}^1(C)
 \simeq \bigoplus_{a_0 \in \F_q^{\times}}\bigoplus_{a_1 \in \F_q}H_{\rm c}^1(\mathbb{A}^1,\mathscr{L}_{{}^0 \psi}(a_0 x^l+a_1 x)), \\
H^1
(\overline{C}_{a,b}) & \simeq 
H^1_{\rm c}
(C_{a,b})
 \simeq \bigoplus_{c \in \F_p^{\times}}
H_{\rm c}^1(\mathbb{A}^1,\mathscr{L}_{{}^0 \psi}
(cax^l+cb x)).
 \end{align*}
 Thus the claim follows. 
\end{proof}
\subsection{Stickelberger's theorem and Witt sum valuation}
By computing the valuation of $S_{\xi}$ when $q=p=l$, 
we deduce that $\overline{C}_{m,n,p,q}$ is not supersingular 
for $n \ge 2$ and arbitrary $q$.

Let $p\ge 3$ be a prime number. Let 
\[
\xi\colon W_2(\F_p)\to\mathbb{C}^\times
\]
be a faithful character. Let $\psi(x):=\xi(0,x)$ for $x \in \F_p$.

For an integer $n \ge 1$, let $\zeta_n:=e^{2\pi i/n}$.  
Fix a prime ideal $\mathfrak{p}_0$ of $\mathbb{Q}(\zeta_{p-1})$ lying above $p$. Then the mod $\mathfrak{p}_0$ reduction induces an isomorphism 
\[
\mu_{p-1}(\mathbb{C})\to\F_p^\times. 
\]
Let $\nu$ denote the character 
\[
\F_p^\times\xrightarrow{\cong}\mu_{p-1}(\mathbb{C})\hookrightarrow \mathbb{C}^\times. 
\]
We set $\nu(0)=0$. 
Set 
\begin{align*}
S_\xi&:=\sum_{x\in\F_p}\xi(x,0),\\
U_\xi&:=\sum_{x\in\F_p}\nu(x)^{-1}\xi^{p-1}(x,0). 
\end{align*}

Let $\mathfrak{p}'$ be the unique prime ideal of $\mathbb{Q}(\zeta_{p-1},\zeta_{p^2})$ lying above $\mathfrak{p}_0$. 
\begin{proposition}\label{vals}
Let $v$ denote the valuation of $\mathbb{Q}(\zeta_{p-1},\zeta_{p^2})$ associated to $\mathfrak{p}'$, normalized by $v(p)=1$. The following hold: 
    \begin{enumerate}
     \item $v(S_\xi U_\xi)=\dfrac{1}{p-1}. $
        \item $v(S_\xi)=\dfrac{1}{p}. $
        \item $v(U_\xi)=\dfrac{1}{p(p-1)}. $
    \end{enumerate}
\end{proposition}
\begin{proof}
$(1)$: 
We simply write $g(x,y)$ for $g_p(x,y)$. 
We compute 
\begin{align*}
S_\xi U_\xi&=\sum_{x,y\in\F_p}\xi(x,0)\nu(y)^{-1}\xi((p-1)(y,0))\\
&=\sum_{x,y\in\F_p}\nu(y)^{-1}\xi((x,0)+(-y,y^p))\\
&=\sum_{x,y\in\F_p}\nu(y)^{-1}\xi(x-y,y^p+g(x,-y))\\
&=\sum_{y,z\in\F_p}\nu(y)^{-1}\xi(z,y^p-g(y,z))\quad (z:=x-y)\\
&=\sum_z\xi(z,0)\sum_y\nu(y)^{-1}\psi(y^p-g(y,z)). 
\end{align*}
We claim that $v(\sum_y\nu(y)^{-1}\psi(y^p-g(y,z)))\geq \frac{1}{p-1}$ for any $z$. Indeed, we have 
\[
\sum_y\nu(y)^{-1}\psi(y^p-g(y,z)))\equiv \sum_y\nu(y)^{-1}=0\quad(\text{mod $\mathfrak{p}'$}). 
\]
Since the sum is contained in $\mathbb{Q}(\zeta_{p-1},\zeta_p)$, we must have $v(\sum_y\nu(y)^{-1}\psi(y^p-g(y,z)))\geq \frac{1}{p-1}$. 

For $x,\,y \in \mathbb{Q}(\zeta_{p-1},\zeta_{p^2})$
and a rational number $r$, we write $x \equiv y \pmod {r+}$ if 
$v(x-y)>r$. 
We can proceed 
\begin{align*}
    S_\xi U_\xi&=\sum_z\xi(z,0)\sum_y\nu(y)^{-1}\psi(y^p-g(y,z))\\
    &\equiv\sum_{z,y}\nu(y)^{-1}\psi(y^p-g(y,z))\quad (\text{mod $\frac{1}{p-1}+$})\\
    &=\sum_{z,y}\nu(y)^{-1}\psi(y^p(1-g(1,z)))\quad (\text{replace $z$ by $yz$})\\
\end{align*}
    if $1=g(1,z)$, then the sum $\sum_y$ is zero. If $1\neq g(1,z)$, we have
    \begin{align*}
        \sum_{y}\nu(y)^{-1}\psi(y^p(1-g(1,z)))&=\sum_{y}\nu(y)^{-1}\psi(y(1-g(1,z))) \quad(q=p)\\
        &=\nu(1-g(1,z))\sum_{y}\nu(y)^{-1}\psi(y). 
    \end{align*}
Set $\alpha:=\sum_{y}\nu(y)^{-1}\psi(y)$. By Stickelberger's formula for $p$-adic valuation of Gauss sum, we have $v(\alpha)=\frac{1}{p-1}$. 

We obtain 
\begin{align*}
S_\xi U_\xi&\equiv\sum_z\nu(1-g(1,z))\alpha \quad(\text{mod $\frac{1}{p-1}+$}). 
\end{align*}
To conclude, it remains to show that $\sum_z\nu(1-g(1,z))$ is a $\mathfrak{p}'$-adic unit. This follows from the computation 
\begin{align*}
    \sum_z\nu(1-g(1,z))\equiv \sum_z(1-g(1,z))\quad(\text{mod $\mathfrak{p}_0$})
\end{align*}
and the observations that $\sum_zz^i=0$ for $i=0,1,\dots, p-2$ and that $g(1,z)$ is of the form $-z^{p-1}+(\text{lower term})$. 
Hence 
\[
\sum_z(1-g(1,z))\equiv -\sum_z z^{p-1} \equiv 1
\quad(\text{mod $\mathfrak{p}_0$}). 
\]
Thus the claim follows. 

$(2), (3)$: We show 
\[
v(S_\xi)\geq \frac{1}{p}, \qquad v(U_\xi)\geq \frac{1}{p(p-1)}. 
\]
Then the claims follow from $(1)$. For the estimate of $S_\xi$, 
\[
S_\xi =\sum_x\xi(x,0)\equiv 0\quad(\text{mod $\mathfrak{p}'$}). 
\]
Let $K$ be the unique subextension of $\mathbb{Q}(\zeta_{p^2})/\mathbb{Q}$ such that $[K:\mathbb{Q}]=p$. Note that the subgroup $\Gal(\mathbb{Q}(\zeta_{p^2})/K)\subset \Gal(\mathbb{Q}(\zeta_{p^2})/\mathbb{Q})\cong W_2(\F_p)^\times$ consists of elements of the form $(a,0)$. 
The action of $(a,0)$ is given by $\xi(1,0) \mapsto 
\xi(a,0)$. Then $\xi(x,0)$ is sent to $\xi(ax,0)$ by the action of $(a,0)$ for $x \in \F_p$. 
Hence $S_\xi\in K$. This proves $v(S_\xi)\geq \frac{1}{p}$ since the ramification index of $K/\mathbb{Q}$ at $p$ is $p$.

Similarly, 
\[
U_\xi=\sum_{x \in \F_p} \nu(x)^{-1} 
\xi^{p-1}(x,0)\equiv \sum_{x \in \F_p} 
\nu(x)^{-1}=0 \quad(\text{mod $\mathfrak{p}'$}). 
\]
Thus $v(U_{\xi}) \geq (p(p-1))^{-1}$. 
%The estimate for $U_\xi$ follows similarly. 
\end{proof}
\begin{corollary}\label{nss}
Assume that $p$ is odd and $n \ge 2$. 
Then the curve $\overline{C}_{m,n,p,q}$ is 
not supersingular. 
\end{corollary}
\begin{proof}
By Lemma \ref{ff}, 
we have finite morphisms 
\[
C_{m,n,p,q} \xrightarrow{\phi_1} C_{1,n,p,q} \xrightarrow{\phi_2} 
C_{1,n,p,p} \xrightarrow{\phi_3} C_{1,2,p,p}. 
\]
This induces a finite morphism $\overline{C}_{m,n,p,q} \to \overline{C}_{1,2,p,p}$. 
Hence it suffices to show the claim for 
$m=1$, $n=2$ and $q=p$. 
If $\overline{C}_{1,2,p,p}$ is supersingular, 
the valuation of the Frobenius trace $-S_{\xi}$ is greater than or equal to $1/2$. 
This does not occur by Proposition \ref{vals}. 
\end{proof}
\begin{remark}
The above corollary can also be deduced from using \cite[Theorem 1.3]{Li}. 
Here we give a short proof for completeness.
\end{remark}

\begin{lemma}
Write $q=p^f$ with a prime $p$, $l$ an odd prime, and assume $p \equiv 1 \pmod{l}$. 
\begin{itemize}
\item[{\rm (1)}] 
The projective curve defined by $y^p - y = x^l$
is not supersingular.
\item[{\rm (2)}] 
The curve $\overline{C}_{m,n,l,q}$ is not supersingular.
\end{itemize}
\end{lemma}

\begin{proof}
(1) 
For $\chi \in \mu_l(\F_p)^{\vee} \setminus \{1\}$ and $\psi \in \F_p^{\vee} \setminus \{1\}$, 
the Gauss sum
$G(\chi,\psi)=\sum_{x \in \F_p^{\times}}
\chi(x^{(p-1)/l}) \psi(x)$ is not of the form $\sqrt{p}$ times a root of unity
by Stickelberger's theorem (cf.\ \cite[Remark 2.2]{A}). 
Let $D$ be the projective curve with affine equation $y^p-y=x^l$. 
Then $G(\chi,\psi)$ appears as a Frobenius eigenvalue on $H^1(D)$. 
Hence the claim follows.
 
(2) 
By Lemma \ref{ff}, we have a finite morphism  $C_{m,n,l,q} \to C_{1,2,l,q}$ and the curve 
$C_{1,2,l,q}$ is defined by $y^q-y=(x^q-x)^l$ by Lemma \ref{de}. 
Further, we have a finite morphism 
$\overline{C}_{1,2,l,q} \to D$ given by 
$(x,y) \mapsto (x^q-x, \sum_{i=0}^{f-1} y^{p^i})$. 
Hence the claim follows from (1). 
\end{proof}

\subsection{Cubic case}
In this section, we always assume that $m=1$, $n=2$ and $l=3$. Then
$g_3(x,x')=-xx'(x+x')$. 
We simply write $C$ for $C_{1, 2, 3,q}$. 
Hence the curve  
$C$ is defined by
\[
y^q-y=-g_3(x^q,-x)=-x^{q+1}(x^q-x). 
\]
For a finite extension of finite fields 
$\F_{p^t}/\F_{p^s}$, let $\Tr_{p^t/p^s}
\colon \F_{p^t} \to \F_{p^s}$ denote
the trace map.

\subsection{Characteristic two}
First, we treat the case $p=2$. 
 \begin{lemma}
Assume $p=2$. 
For $(a,b) \in \F_q^{\times} \times \F_q$, 
let $E_{a,b}$ denote the 
elliptic curve over $\F_q$ defined by 
$y^2+y=ax^3+bx$. 
Then 
\begin{align*}
L_{\overline{C}/\F_q}(T)&=\prod_{(a,b) \in \F_q^{\times} \times \F_q} 
L_{E_{a,b}/\F_q}(T) \\
&=\prod_{(a,b) \in \F_q^{\times} \times \F_q}
\biggl(1+\sum_{x \in \F_q} (-1)^{\Tr_{q/2}(ax^3+bx)} T+qT^2\biggr) \\
&=\prod_{\xi \in W_{\rm prim}^{\vee}}
\left(1+S_{\xi} T+qT^2\right).    
\end{align*}
Hence $\overline{C}$ is supersingular.  
 \end{lemma}
 \begin{proof}
The first equality 
follows from Lemma \ref{coho}. 
The second equality is well-known. 
The third equality follows from Lemma \ref{sum}. 

Since each $E_{a,b}$ is supersingular, 
the curve $\overline{C}$, whose $L$-polynomial is the product of 
$L_{E_{a,b}/\F_q}(T)$, is also supersingular.
 \end{proof}
 \begin{remark}
 The sum 
 $\sum_{x \in \F_q} (-1)^{\Tr_{q/2}(ax^3+bx)}$
 is explicitly calculated in \cite{Ca2}. 
 \end{remark}
 \subsection{Odd characteristic}
 Through this subsection, we assume $p \neq 2$. 
 We count the number of fixed points. 
\begin{lemma}\label{16}
Let $(a,b) \in W$. 
Then we have 
\[
\# C_{\F}^{\Fr_{q^2} \circ (a,b)} =
\begin{cases}
2q^2, & \text{if } a \neq 0 \text{ and } b=-\dfrac{a^3}{4}+\zeta a 
\text{ with } \zeta \in \F_q^{\times} \setminus (\F_q^{\times})^2, \\[7pt]
q^2, & \text{if } a \neq 0 \text{ and } b=-\dfrac{a^3}{4}, \\[4pt]
q^3, & \text{if } a=b=0, \\[4pt]
0, & \text{otherwise.}
\end{cases}
\]
\begin{comment}
Assume $2 \mid q$. 
We have 
\[
\left|C_{\F}^{\Fr_{q^2} \circ (a,b)}\right|
=\begin{cases}
2q^2 & \textrm{if $a \neq 0$ and 
$\Tr_{q/2}(b/a^3)=1$}, \\
q^2 & \textrm{if $a \neq 0$ and $\Tr_{q/2}(b/a^3)=0$},  \\
q^3 & \textrm{if $a =b=0$}, \\
0 &  \textrm{otherwise}. 
\end{cases}
\]
\end{comment}
\end{lemma}
\begin{proof}
By \eqref{xycd}, 
it suffices to count the number of pairs $(x,y) \in \F^2$ satisfying
\begin{gather}\label{sol}
\begin{aligned}
&x^{q^2} +a= x, \\
&y^{q^2} - a x^{q^2}(x^{q^2}+a) + b = y, \\
&y^q - y = - x^{q+1}(x^q - x).
\end{aligned}
\end{gather}
Substituting the first equation into the second gives
\[
y^{q^2} - y = a x^{q^2}(x^{q^2}+a) - b = ax^2-a^2 x - b.
\]
Using the first and third equations of \eqref{sol}, we obtain
\begin{align*}
y^{q^2} - y 
&= (y^q - y)^q + (y^q - y) \\
&= -x^{q^2+q}(x^{q^2}-x^q) - x^{q+1}(x^q - x) \\
&=- x^q(x-a)(x - x^q - a) + x^{q+1}(x^q - x) \\
&= -a x^{2q} + 2a x^{q+1} - a^2 x^q.
\end{align*}
Hence, 
\[
a x^{2q} - 2a x^{q+1} + a^2 x^q+ ax^2- a^2 x = b.
\]
Set $z := x^q - x$. 
Then this implies
\[
a z^2 + a^2 z = b.
\]
From $x^{q^2} - x = -a$, we also have
\[
z^q + z = x^{q^2} - x = -a.
\]
Therefore \eqref{sol} is equivalent to
\begin{gather}\label{soll}
\begin{aligned}
&z = x^q - x, \\
&z^q + z = -a, \\
&a z^2 + a^2 z = b, \\
&y^q - y = -x^{q+1}(x^q - x).
\end{aligned}
\end{gather}

\smallskip
\noindent
\textbf{Case 1:} $a \neq 0$.  
Let $z_1 = z + (a/2)$.  
Substituting this into $a z^2 + a^2 z = b$ gives
\begin{equation}\label{z1}
z_1^2 = \frac{a^2}{4} + \frac{b}{a}.
\end{equation}
The equality $z^q + z = -a$ is equivalent to $z_1^q = -z_1$.

\smallskip
\emph{{\rm (i)} Assume $b \neq -a^3/4$.}  
Then $z_1 \neq 0$, and hence $z_1^{q-1} = -1$.  
Let $\nu_q(x) = x^{(q-1)/2}$ for $x \in \F_q^{\times}$.  
From \eqref{z1}, we have
\[
\nu_q\!\left(\frac{a^2}{4} + \frac{b}{a}\right)
= z_1^{q-1} = -1.
\]
Thus, if a solution of \eqref{soll} exists, we must have 
$\dfrac{a^2}{4} + \dfrac{b}{a} \in 
\F_q^{\times} \setminus (\F_q^{\times})^2$.  
If this condition holds, there exist $2q^2$ solutions of \eqref{soll}.

\smallskip
\emph{{\rm (ii)} Assume $b =- a^3/4$.}  
Then $z_1 = 0$ by \eqref{z1}, and clearly $z_1^q = z_1$.  
Hence there are $q^2$ solutions of \eqref{soll}$.$

\smallskip
\noindent
\textbf{Case 2:} $a = 0$.  
If $b \neq 0$, there is no solution of \eqref{soll}.  
Assume $b = 0$.  
Then we consider
\[
x^{q^2} = x, \qquad y^{q^2} = y, \qquad 
y^q - y = -x^{q+1}(x^q - x).
\]
If $x = 0$, there are $q$ solutions.  
Assume $x \neq 0$.  
Since $x^{q+1} \in \F_q^{\times}$, we have
\[
\left(\frac{y}{x^{q+1}}\right)^q - 
\frac{y}{x^{q+1}} = -x^q + x.
\]
Thus $(y/x^{q+1}) + x \in \F_q$, and hence $y \in \F_{q^2}$.  
Namely, there are $q(q^2 - 1)$ such solutions.  
Therefore we obtain $q^3$ solutions of \eqref{soll}.
\end{proof}
In the following, let $\xi \in W_{\mathrm{prim}}^{\vee}$. 
\begin{corollary}\label{cc}
Let $\nu_q(x) := x^{(q-1)/2}$ for $x \in \F_q^{\times}$, and define
\[
U_{\xi} := \sum_{x \in \F_q^{\times}} \nu_q(x)\, \xi^2(x,0).
\]
Let $\psi(x) := \xi(0,x)$ for $x \in \F_q$, and put
\[
G(\psi) := \sum_{x \in \F_q} \psi(x^2).
\]
Then
\[
\Tr(\Fr_{q^2}^*; H_{\mathrm{c}}^1(\mathbb{A}^1, \mathscr{L}_{\xi}))
= -q + \nu_q(2)\, G(\psi)\, U_{\xi}.
\]
\end{corollary}

\begin{proof}
Using Lemma~\ref{16} and the fixed point formula (cf.\ \cite[Section~3]{DL}), we have
\begin{gather}\label{sol2}
\begin{aligned}
& \Tr(\Fr_{q^2}^*; H_{\mathrm{c}}^1(\mathbb{A}^1, \mathscr{L}_{\xi})) \\
&= \frac{1}{q^2}\sum_{(a,b)\in W}
\xi^{-1}(a,b)\Tr\bigl((\Fr_{q^2}\circ(a,b))^*; H_{\mathrm{c}}^1(C)\bigr) \\
&= \frac{1}{q^2}\sum_{(a,b)\in W}
\xi^{-1}(a,b)\bigl(q^2 - \#C_{\F}^{\Fr_{q^2}\circ(a,b)}\bigr) \\
&= -\sum_{x\in\F_q^{\times}}\xi^{-1}\left(x,-\frac{x^3}{4}\right)
 -2\sum_{x\in\F_q^{\times},\, \zeta\in\F_q^{\times}\setminus(\F_q^{\times})^2}
 \xi^{-1}\left(x,-\frac{x^3}{4}+\zeta x\right) - q \\
&= -\sum_{x\in\F_q^{\times}}\xi^{-2}\left(\frac{x}{2},0\right)
 -2\sum_{x\in\F_q^{\times}}\xi^{-2}\left(\frac{x}{2},0\right)
 \sum_{\zeta\in\F_q^{\times}\setminus(\F_q^{\times})^2}\xi(0,-\zeta x)
 - q,
\end{aligned}
\end{gather}
where we used $2 (x/2,0)= (x, -x^3/4)$ in the last equality.

For $x \in \F_q^{\times}$, we have
\[
\sum_{\zeta\in\F_q^{\times}\setminus(\F_q^{\times})^2}\xi(0,-\zeta x)
+ \frac{1}{2}\sum_{\zeta\in\F_q^{\times}}\xi(0,-\zeta^2 x)
= \sum_{\zeta\in\F_q^{\times}}\xi(0,-\zeta x) = -1,
\]
and therefore
\[
\sum_{\zeta\in\F_q^{\times}\setminus(\F_q^{\times})^2}\xi(0,-\zeta x)
= -\frac{1}{2}\left(\sum_{\zeta\in\F_q}\xi(0,-\zeta^2 x) + 1\right)
= -\frac{\nu_q(-x) G(\psi) + 1}{2}.
\]
Using this, we obtain
\begin{align*}
&2\sum_{x\in\F_q^{\times}}\xi^{-2}\left(\frac{x}{2},0\right)
  \sum_{\zeta\in\F_q^{\times}\setminus(\F_q^{\times})^2}\xi(0,-\zeta x) \\
&= -\sum_{x\in\F_q^{\times}}
 \xi^{-2}\left(\frac{x}{2},0\right)\bigl(\nu_q(-x)G(\psi) + 1\bigr) \\
&= -G(\psi)\sum_{x\in\F_q^{\times}}\nu_q(-x)\xi^{-2}\left(\frac{x}{2},0\right)
   -\sum_{x\in\F_q^{\times}}\xi^{-2}\left(\frac{x}{2},0\right) \\
&= -\nu_q(2)\, G(\psi)\, U_{\xi}
   -\sum_{x\in\F_q^{\times}}\xi^{-2}\left(\frac{x}{2},0\right).
\end{align*}
Substituting this into \eqref{sol2} yields the claimed formula.
\end{proof}

\begin{lemma}\label{2}
We have
$S_{\xi}^2 = q + \nu_q(2)\, G(\psi)\,U_{\xi}$. 
\end{lemma}

\begin{proof}
Note that $2(x/2,0)=(x,-x^3/4)$.  We compute
\begin{align*}
S_{\xi}^2
=\Bigl(\sum_{x\in\F_q}\xi(x,0)\Bigr)^2
&=\sum_{x,y\in\F_q}\xi(x+y,- xy(x+y)) \\
&=\sum_{x,z\in\F_q}\xi\bigl(z,\;- x(z-x)z\bigr)\qquad(z:=x+y)\\
&=\sum_{x,z\in\F_q}\xi\left(z,\;z\left(x-\frac{z}{2}\right)^2-\frac{z^3}{4}\right)\\
&=\sum_{x_1,z\in\F_q}\xi\left(z,\;z x_1^2-\frac{z^3}{4}\right)
\qquad\left(x_1:=x-\frac{z}{2}\right)\\
&=\sum_{z\in\F_q}\xi^2\left(\frac{z}{2},0\right)\sum_{x_1\in\F_q}\psi(z x_1^2).
\end{align*}
For $z=0$ the inner sum equals $q$, and for $z\neq0$ we apply the Gauss-sum identity
\(\sum_{t\in\F_q}\psi(\alpha t^2)=\nu_q(\alpha)G(\psi)\) for $\alpha \in \F_q^{\times}$ to get
\[
S_{\xi}^2
= q + \sum_{z\in\F_q^\times}\xi^2\left(\frac{z}{2},0\right)\,\nu_q(z)\,G(\psi).
\]
Putting $w:=z/2$, we obtain
\[
S_{\xi}^2 = q + \nu_q(2)\,G(\psi)\sum_{w\in\F_q^\times}\nu_q(w)\xi^2(w,0)
= q + \nu_q(2)\,G(\psi)\,U_{\xi}.
\]
This completes the proof.
\end{proof}

The following formula is a generalization of 
\cite[(1.3)]{Ca}. 
\begin{proposition}\label{3}
We let 
\[
T_{\xi}:=\sum_{x \in \F_q} \xi(x^2,0). 
\]
Then 
\[
S_{\xi}^2=q+\nu_q(2)\, G(\psi) \left(T_{\xi^2}-S_{\xi^2}\right). 
\]
\end{proposition}
\begin{proof}
We have 
\begin{align*}
T_{\xi}=1+\sum_{x \in \F_q^{\times}} \xi(x^2,0)=1+
\sum_{x \in \F_q^{\times}} (1+\nu_q(x))
\xi(x,0)=S_{\xi}+\sum_{x \in 
\F_q^{\times}}\nu_q(x) \xi(x,0). 
\end{align*}
Replacing $\xi$ by $\xi^2$ gives 
\[
U_{\xi}=\sum_{x \in 
\F_q^{\times}}\nu_q(x) \xi^2(x,0)=T_{\xi^2}-S_{\xi^2}. 
\]
Thus Lemma \ref{2} implies 
\[
S_{\xi}^2=q+\nu_q(2)\, G(\psi)\,
U_{\xi}
=q+\nu_q(2) \,G(\psi)\,(T_{\xi^2}-S_{\xi^2}). 
\]
Hence we obtain the claim. 
\end{proof}
We recall $\dim H_{\rm c}^1(\mathbb{A}^1,\mathscr{L}_{\xi})=2$ by Lemma \ref{dim}. 
\begin{corollary}\label{cc2}
Let $\alpha_{\xi},\, \beta_{\xi}$ be the eigenvalues of $\Fr_q^\ast$ on 
$H_{\rm c}^1(\mathbb{A}^1,\mathscr{L}_{\xi})$. 
Then we have 
$\alpha_{\xi}+\beta_{\xi}=-S_{\xi}$
and $\alpha_{\xi} \beta_{\xi}=q$. 
\end{corollary}
\begin{proof}
By lemma \ref{trace}(1) and Corollary \ref{cc}, we have 
\begin{align*}
\alpha_{\xi}+\beta_{\xi}&=-\sum_{x \in \F_q}\xi(x,0)=-S_{\xi}, \\
\alpha_{\xi}^2+\beta_{\xi}^2&=-q+\nu_q(2)\,G(\psi)\,U_{\xi}. 
\end{align*}
Hence 
\begin{align*}
2\alpha_{\xi}\beta_{\xi}=(\alpha_{\xi}+\beta_{\xi})^2-(\alpha_{\xi}^2+\beta_{\xi}^2) =S_{\xi}^2+q-\nu_q(2)\,G(\psi)\,U_{\xi}=2q, 
\end{align*}
where the last equality follows from Lemma \ref{2}.  
\end{proof}
We give an explicit formula for the $L$-polynomial $L_{\overline{C}/\F_q}(T)$. 
\begin{theorem}\label{main}
For $\xi \in W_{\mathrm{prim}}^{\vee}$, let 
$S_{\xi}=\sum_{x \in \F_q} \xi(x,0)$. 
Then we have 
\[
L_{\overline{C}/\F_q}(T)=\prod_{\xi \in W_{\rm prim}^{\vee}}
\left(1+S_{\xi} T+qT^2\right). 
\]
\end{theorem}
\begin{proof}
By the proof of Corollary \ref{genus} and Lemma \ref{Fum}, we have 
\[
H^1(\overline{C}) \xleftarrow{\sim} 
H_{\mathrm{c}}^1(C)\simeq \bigoplus_{\xi\in W_{\mathrm{prim}}^{\vee}}
H_{\mathrm{c}}^1(\mathbb{A}^1,\mathscr{L}_{\xi}).
\]
Therefore the determinant factors as the product of the determinants on the summands:
\[
L_{\overline{C}/\F_q}(T)
=\prod_{\xi\in W_{\mathrm{prim}}^{\vee}}
\det\!\big(1-\Fr_q^\ast T; H_{\mathrm{c}}^1(\mathbb{A}^1,\mathscr{L}_{\xi})\big).
\]

For each $\xi\in W_{\mathrm{prim}}^{\vee}$ let $\alpha_{\xi},\beta_{\xi}$ be the Frobenius eigenvalues on $H_{\mathrm{c}}^1(\mathbb{A}^1,\mathscr{L}_{\xi})$. By Corollary~\ref{cc2} they satisfy
\[
\alpha_{\xi}+\beta_{\xi}=-S_{\xi},\qquad
\alpha_{\xi}\beta_{\xi}=q.
\]
Hence the local factor for $\xi$ is
\[
\det\!\big(1-\Fr_q^\ast T; H_{\mathrm{c}}^1(\mathbb{A}^1,\mathscr{L}_{\xi})\big)
=(1-\alpha_{\xi}T)(1-\beta_{\xi}T) 
=1+S_{\xi}T+qT^2. 
\]

Combining these factors for all $\xi\in W_{\mathrm{prim}}^{\vee}$ yields the claimed  formula.
\end{proof}
\begin{corollary}
For $\xi \in W_{\rm prim}^{\vee}$, 
let $0 \le \theta_{\xi} \le \pi$ be the element satisfying 
\[
\cos \theta_{\xi}=-\frac{S_{\xi}}{2 \sqrt{q}} 
\]
(cf.\ \eqref{abs}). 
Then we have 
\[
\# \overline{C}(\F_{q^n})
=q^n+1-q^{n/2}\sum_{\xi \in W_{\rm prim}^{\vee}}
2 \cos n \theta_{\xi}.  
\]
\end{corollary}
\begin{proof}
Let the notation be as in the proof of Theorem 
\ref{main}. 
We have $\{\alpha_{\xi},\beta_{\xi}\}
=\{\sqrt{q} e^{\pm i \theta_{\xi}}\}$. 
Hence the claim follows from
\[
\# \overline{C}(\F_{q^n})
=q^n+1-\sum_{\xi \in W_{\rm prim}^{\vee}}
(\alpha_{\xi}^n+\beta_{\xi}^n)
=q^n+1-q^{n/2}\sum_{\xi \in W_{\rm prim}^{\vee}}
2 \cos n \theta_{\xi}. 
\]
\end{proof}
\begin{comment}
\begin{corollary}
We write $q=p^f$ with a prime number $p$. 
We assume $p \neq 3$.
Let ${}^0 \psi \in \F_p^{\vee} \setminus \{1\}$, 
let $\psi_q:={}^0 \psi \circ \Tr_{q/p}$
and let 
\[
B(a,b,{}^0 \psi):=-\sum_{x \in \F_q}\psi_q(ax^3+bx)
\] 
for $a \in \F_q^{\times}$ and $b \in \F_q$. 
Let $0 \le \theta_{a,b,{}^0 \psi} \le \pi$ be the element satisfying 
\[
\cos \theta_{a,b,{}^0 \psi}=\frac{B(a,b,{}^0 \psi)}{2 \sqrt{q}}. 
\]
We have 
\[
|\overline{C}(\F_{q^n})|
=q^n+1-q^{n/2}\sum_{a\in \F_q^{\times},\, b\in \F_q}
2 \cos n \theta_{a,b,{}^0 \psi}.  
\]
\end{corollary}
\end{comment}
\begin{corollary}
Assume that $q \equiv 2 \pmod 3$ and 
that $p$ is odd. 
Let $\psi \in \F_q^{\vee} \setminus \{1\}$,
and define 
\[
S_{\psi,a}:=\sum_{x \in \F_q}\psi(x^3+ax).
\]
Then 
\[
L_{\overline{C}/\F_q}(T)
=\prod_{a \in \F_q}
(1+S_{\psi,a} T+q T^2)^{q-1}.
\]
\end{corollary}

\begin{proof}
Let $\xi \in W_{\mathrm{prim}}^{\vee}$. 
By Lemma~\ref{sum}, there 
exists $(a,b) \in \F_q^{\times} \times \F_q$ such that  
\[
S_{\xi}=\sum_{x \in \F_q} \psi(ax^3+bx). 
\]
Since $q \equiv 2 \pmod 3$, 
we can take $c \in \F_q^{\times}$ such that $c^3=a$. 
Letting $d:=a^{-1}b$, we then have 
$S_{\xi}=S_{\psi,d}$. 
The assertion follows from Theorem~\ref{main}.
\end{proof}
\subsubsection{Katz's monodromy theory and the supersingularity of $\overline{C}$}

Let $\mathscr{F}$ be a smooth $\overline{\mathbb{Q}}_{\ell}$-sheaf on $\mathbb{A}^1_{\F_q}$. 
Let 
$\rho_{\mathscr{F}}$ denote the $\ell$-adic representation of 
$\pi_1(\mathbb{A}^1_{\F_q})$ associated to 
$\mathscr{F}$. 
The Zariski closure of the image of $\pi_1(\mathbb{A}^1_{\F_q})$ (resp.\ the normal subgroup $\pi_1(\mathbb{A}^1)$) by $\rho_{\mathscr{F}}$
is called the arithmetic (resp.\ geometric) monodromy group of $\mathscr{F}$, we denote by $G_{\mathscr{F}, \rm arith}$ (resp.\ $G_{\mathscr{F}, {\rm geom}}$). 
If $G_{\mathscr{F}, {\rm arith}}$
is finite, $\mathscr{F}$ is said to have finite monodromy. 
%It is known that $\mathscr{F}_{\rm arith}$is finite if and only if $\mathscr{F}_{\rm geom}$is finite in \cite{Ka}. 

Choose and fix a square root $\sqrt{q}\in\overline{\mathbb{Q}}_{\ell}^\times$.
Let $\Gal(\F/\F_q)$ be the Galois group. 
Define the one-dimensional Galois representation $\overline{\mathbb{Q}}_{\ell}\bigl(\tfrac12\bigr)$
by requiring that a geometric Frobenius $\mathrm{Frob}_q \in \Gal(\F/\F_q)$ acts on it as $\sqrt{q}^{-1}$.

Let $l$ be a prime and  
let $\psi \in \F_q^{\vee} \setminus \{1\}$. 
Let $\pi \colon \mathbb{A}^2_{\F_q}
\to \mathbb{A}_{\F_q}^1,\ (x,t) \mapsto t$.  
We define a local system on $\mathbb{A}_{\F_q}^1$: 
\[
\mathscr{F}_{\psi,l,q}:=R^1\pi_! 
\mathscr{L}_{\psi}(x^l+tx)\bigl(\tfrac12\bigr) 
\]
(cf.\ \cite{KL} or \cite{La}). 
\begin{proposition}
Assume $p \neq l$. 
The curve $\overline{C}_{1,2,l,q}$ is supersingular for all
$q$ if and only if $\mathscr{F}_{{}^0 \psi \circ \Tr_{q/p},l,q}$
has finite monodromy for ${}^0 \psi \in \F_p^{\vee} \setminus \{1\}$. 
\end{proposition}
\begin{proof}
We write $C_{l,q}$ for $C_{1,2,l,q}$. 
Let $t \in \F_q$. 
By the proper base change theorem, 
\[
(\mathscr{F}_{\psi,l,q})_{\bar{t}}
\simeq H_{\rm c}^1(\mathbb{A}^1,\mathscr{L}_{\psi}(x^l+tx))\bigl(\tfrac12\bigr). 
\]
Hence, by the proof of Lemma \ref{coho}, 
\[
H^1(\overline{C}_{l,q})\bigl(\tfrac12\bigr) \simeq 
\bigoplus_{\psi \in \F_q^{\vee} \setminus \{1\}}
\bigoplus_{t \in \F_q} (\mathscr{F}_{\psi,l,q})_{\bar{t}}. 
\]
The geometric monodromy group of 
$\mathscr{F}_{\psi,l,q}$ equals the one of $\mathscr{F}_{{}^0 \psi \circ \Tr_{q/p},l,q}$. 
Hence the claim follows from \cite[Theorem 8.14.4]{Ka}
(cf.\ \cite[Lemma 1 and Corollary 1]{R}). 
\end{proof}

The \emph{period} of a supersingular curve $X$ over $\F_q$
is defined as the smallest positive integer $n$ such that
the Frobenius map $\Fr_{q^n}^\ast$
acts on $H^1(X)$ as multiplication by $q^{n/2}$. 

\begin{lemma}
Assume $q=5^f$. 
Then the curve $\overline{C}$ is supersingular with period $60$. 
\end{lemma}

\begin{proof}
%The supersingularity follows from \cite{KT} (cf.\ \cite[Theorem 8.14.5]{Ka} and \cite[Theorem 1]{GR}). 

Let ${}^0 \psi \in \F_5^{\vee} \setminus \{1\}$. 
By \cite[Theorem 17.1(1)]{KT0}, both the arithmetic and geometric monodromy groups of $\mathscr{F}_{{}^0 \psi \circ \Tr_{q/5},3,q}$ are $\mathrm{SL}_2(\F_5)$. 
The least common multiple of the orders of elements of  $\mathrm{SL}_2(\F_5)$ is $60$. 
Hence the required assertion follows. 
\end{proof}
\begin{comment}
\begin{example}
Let $\zeta_n=e^{2\pi i/n}$ for an integer $n \ge 1$.
Let $q=5$. We observe that the period of $\overline{C}$
equals $60$. 
Let $S_{\psi,a} = \sum_{x \in \F_5} \psi(x^3+ax)$ with $\psi(y) = \zeta_5^y$.
\begin{table}[htbp]
\centering
\caption{Values of $S_{\psi,a}$ and roots of $f_a(t)=t^2+S_{\psi,a}t+5$ for $a \in \F_5$.}
\vspace{1ex}
\begin{tabular}{c|c|c}
\hline
$a$ & $S_{\psi,a}$ & roots of $f_a(t)=t^2+S_{\psi,a}t+5$ \\[0.5mm]
\hline 
$0$ & $0$ & $\sqrt{5}\,\zeta_4^{\pm 1}$ \\[1mm]
$1$ & $\displaystyle \frac{5-\sqrt{5}}{2}$ & $\sqrt{5}\,\zeta_{10}^{\pm 3}$ \\[1mm]
$2$ & $-\sqrt{5}$ & $\sqrt{5}\,\zeta_6^{\pm 1}$ \\[1mm]
$3$ & $\sqrt{5}$ & $\sqrt{5}\,\zeta_3^{\pm 1}$ \\[1mm]
$4$ & $\displaystyle \frac{5+\sqrt{5}}{2}$ & $\sqrt{5}\,\zeta_5^{\pm 2}$ \\[1mm]
\hline
\end{tabular}
\label{tab:Spsi_roots}
\end{table}

\end{example}
\end{comment}

%\begin{question}Is the curve $\overline{C}$ supersingular if and only if $p=2$ or $5$?\end{question}

%\begin{remark}It seems that the curve $\overline{C}$ is not supersingular if $p \neq 2,5$. This observation is supported by the results of \cite[Corollary~3]{R} for $d=3$ and \cite[Proposition~5]{Ka0} (cf.\ \cite{GR} and \cite{KT}).  \end{remark}

\section{Characteristic three}\label{s3.5}
We assume $p=3$ and let $q=3^f$. 
Let $\Tr_{q/3} \colon W_{2,q} \to W_{2,3}$ denote the 
trace map. 
\begin{lemma}\label{ss}
Let ${}^0 \xi \in W_{2,3}^{\vee}$ be a 
faithful character. 
Let $(a,b) \in \F_q^{\times} \times \F_q$, and put $c:=a^{-1}b^{1/3}$. 
Then 
\[
S_{\xi_{a,b}}
=\sum_{x \in \F_q} 
{}^0 \xi\bigl(\Tr_{q/3}(x,cx)\bigr). 
\]
\end{lemma}

\begin{proof}
We have 
\[
S_{\xi_{a,b}}
=\sum_{x \in \F_q}
\xi_{a,b}(x,0)
=\sum_{x \in \F_q} 
{}^0 \xi\bigl(\Tr_{q/3}(ax,bx^3)\bigr)
=\sum_{x \in \F_q} 
{}^0 \xi\bigl(\Tr_{q/3}(ax,b^{1/3}x)\bigr).
\]
Setting $t=ax$, we obtain the claim.
\end{proof}

\begin{corollary}\label{3mainc}
Let ${}^0 \xi \in W_{2,3}^{\vee}$
be a faithful character. 
For $c \in \F_q$, set 
\[
S_{q,c}:=\sum_{x \in \F_q} 
{}^0 \xi\bigl(\Tr_{q/3}(x,cx)\bigr).
\]
Then 
\[
L_{\overline{C}/\F_q}(T)
=\prod_{c \in \F_q} 
(1+S_{q,c}T+qT^2)^{q-1}.
\]
\end{corollary}

\begin{proof}
By Theorem~\ref{main} and Lemmas~\ref{ch} and~\ref{ss}, we obtain 
\begin{align*}
L_{\overline{C}/\F_q}(T)
&=\prod_{\xi \in W_{\mathrm{prim}}^{\vee}} 
(1+S_{\xi}T+qT^2) \\
&=\prod_{(a,b) \in \F_q^{\times} \times \F_q} 
(1+S_{\xi_{a,b}}T+qT^2) 
=\prod_{c \in \F_q} 
(1+S_{q,c}T+qT^2)^{q-1}.
\end{align*}
\end{proof}

\begin{lemma}\label{prod}
Let $c \in \F_q^{\times}$. Let ${}^0 \xi \in W_{2,3}^{\vee}$ be a faithful character
and put $\xi := {}^0 \xi \circ \Tr_{q/3}$. 
Let $\psi(x) := \xi(0,x)$ for $x\in\F_q$, and define
$G(\psi) := \sum_{t\in\F_q}\psi(t^2)$.
Then
\[
S_{q,0}\, S_{q,c}=G(\psi)\sum_{x\in\F_q^{\times}}
\xi\left(x,-(c+1)x-c^2 x^{-1}\right)\nu_q(x).
\]
\end{lemma}

\begin{proof}
We compute
\begin{align*}
S_{q,0}\,S_{q,c}
&=\sum_{x\in\F_q}\xi(x,0)\sum_{y\in\F_q}\xi(y,cy)
=\sum_{x,y}\xi(x+y,\,cy-xy(x+y)) \\
&=\sum_{x,z}\xi\bigl(z,\,c(z-x) - x(z-x)z\bigr)
\qquad (z:=x+y) \\
&=\sum_x \xi(0,-cx)+\sum_x\sum_{z\neq0}
\xi\Bigl(z,\,-z x^2-(z^2+c)x+cz\Bigr),
\end{align*}
where we separated the summand with $z=0$.  The first term vanishes because
$\psi=\xi(0,\cdot)$ is a nontrivial additive character and $c \neq 0$:
\(\sum_x \xi(0,-cx)=\sum_x\psi(-cx)=0\).

For $z\neq0$ complete the square in $x$:
\[
z x^2-(z^2-c)x+cz
= z\Bigl(x-\frac{z^2+c}{2z}\Bigr)^2 -\frac{(z^2-c)^2}{4z}.
\]
Putting $x_1:=x-(2z)^{-1}(z^2-c)$ (which is a bijection of $\F_q$), we get
\[
\sum_{x\in\F_q}\xi\Bigl(z,\,z x^2-(z^2+c)x+cz\Bigr)
=\xi\Bigl(z,-\frac{(z^2-c)^2}{4z}\Bigr)
\sum_{x_1\in\F_q}\psi(z x_1^2).
\]
Using the Gauss-sum identity
\(\sum_{t\in\F_q}\psi(\alpha t^2)=\nu_q(\alpha)\,G(\psi)\) for \(\alpha\in\F_q^\times\),
we obtain for $z\neq0$:
\[
\sum_{x\in\F_q}\xi\Bigl(z,\,z x^2-(z^2+c)x+cz\Bigr)
= \xi\Bigl(z,-\frac{(z^2-c)^2}{4z}\Bigr)\nu_q(z)G(\psi).
\]
Hence
\[
S_{q,0}\,S_{q,c}
= G(\psi)\sum_{z\neq0}\xi\Bigl(z,-\frac{(z^2-c)^2}{4z}\Bigr)\nu_q(z).
\]
Finally note
\[
\Tr_{q/3}\left(\frac{(z^2-c)^2}{z}\right) = 
\Tr_{q/3}\left(z^3+cz+\frac{c^2}{z}
\right)
= \Tr_{q/3}\left((c+1)z+\frac{c^2}{z}\right),
\]
so renaming $z\mapsto x$ yields the claimed formula.
\end{proof}
\subsection{Valuation of the Gauss--Heilbronn sum}
Let $\zeta_9=e^{2\pi i/9}$. 
Let $\mathfrak{p}$ denote the prime of 
$\mathbb{Q}(\zeta_9)$ 
lying over $(3)$. 
Let $v_{\mathfrak{p}}(-)$ denote the valuation of
the completion 
$\mathbb{Q}(\zeta_9)$ at $\mathfrak{p}$ normalized 
such that $v_{\mathfrak{p}}(3)=1$.
For $\xi \in W_{\mathrm{prim}}^{\vee}$, 
we study $v_{\mathfrak{p}}(S_{\xi})$ using supersingular curves. 
\begin{lemma}\label{vv}
Let $(a,b) \in \F_q^{\times} \times \F_q$ and let 
${}^0 \psi \in \F_3^{\vee} \setminus \{1\}$.  
Put $\psi:={}^0 \psi \circ \Tr_{q/3}$ and 
\[
S(a,b,\psi)
:=\sum_{x \in \F_q^{\times}} 
\nu_q(x)\,\psi(ax^3+a x(x+b)b).
\]
Then $v_{\mathfrak{p}}(S(a,b,\psi)) \ge f/2$.
\end{lemma}

\begin{proof}
First assume $b=0$. By $p=3$ and 
$\psi:={}^0 \psi \circ \Tr_{q/3}$,  
we have $\psi(ax^3)=\psi(a^{1/3}x)$. 
Hence 
\[
S(a,0,\psi)
=\sum_{x \in \F_q^{\times}} \nu_q(x)\psi(ax^3)
=\sum_{x \in \F_q^{\times}} \nu_q(x)\psi(a^{1/3}x)
=G(\psi_{a^{1/3}})=\nu_q(a^{1/3})\,G(\psi),
\]
whose $\mathfrak{p}$-adic valuation is $f/2$.

Now assume $b\neq 0$. Writing 
\[
c:=ab \in \F_q^{\times}, \qquad d:=ab^2+a^{1/3},
\]
we have
\[
S(a,b,\psi)
= \sum_{x\in \F_q^{\times}}
\nu_q(x)\,\psi(c x^2 + d x).
\]
Let $D_{c,d}$ be the affine curve defined by 
\[
y^3 - y = c x^4 + d x^2 
\]
over $\F_q$.  
The group $\{\pm1\}\times \F_3$ acts on $D_{c,d}$ via  
\[
(\iota,\zeta)\colon (x,y)\longmapsto (\iota x,\; y+\zeta) \quad \textrm{for $(\iota,\zeta) \in \{\pm 1\} \times \F_3$}.
\]
Let $D^0_{c,d} \subset D_{c,d}$ be the 
open subscheme defined by $x \neq 0$. 
The morphism 
\[
D^0_{c,d} \to \mathbb{G}_{\mathrm{m},\,\F_q}, \quad 
(x,y) \mapsto x^2, 
\]
is a finite \'etale Galois covering 
with Galois group $\{\pm 1\} \times \F_3$.  
Let $\nu_0$ be the nontrivial character of $\{\pm 1\}$, and put 
$\mathbb{G}_{\mathrm{m}}=\Spec \F[t^{\pm1}]$.  
Then
\[
H_{\mathrm{c}}^1(D_{c,d})[\nu_0\boxtimes{}^0 \psi]
\xleftarrow{\sim} 
H_{\mathrm{c}}^1(D^0_{c,d})[\nu_0\boxtimes{}^0 \psi]
\simeq 
H_{\mathrm{c}}^1
\!\left(
\mathbb{G}_{\mathrm{m}},\;
\mathscr{K}_{\nu_0}(t)
\otimes \mathscr{L}_{{}^0 \psi}(c t^2 + d t)
\right),
\]
where $\mathscr{K}_{\nu_0}(t)$ is the Kummer sheaf defined by the Kummer covering 
$y^2=x$ and $\nu_0$.
The sheaf 
$\mathscr{K}_{\nu_0}(t)\otimes \mathscr{L}_{{}^0 \psi}(c t^2 + d t)$ on $\mathbb{G}_{\mathrm{m},\, \F_q}$ 
is tamely ramified at $0$ and wildly ramified at $\infty$ with Swan conductor $2$.  
By the Grothendieck--Ogg--Shafarevich formula,
\[
\dim 
H_{\mathrm{c}}^1
\!\left(
\mathbb{G}_{\mathrm{m}},
\mathscr{K}_{\nu_0}(t)
\otimes \mathscr{L}_{{}^0 \psi}(c t^2 + d t)
\right)
=2.
\]

Since the defining equation of $D_{c,d}$ is
$y^3-y = x(cx^3+d x)$ and the polynomial $cx^3+dx$ is $\F_3$-linearized,
the smooth compactification $\overline{D}_{c,d}$ is 
supersingular by \cite{GV}.  
Hence all eigenvalues of $\Fr_q^{\ast}$ on 
$H^1(\overline{D}_{c,d})$ are of the form 
$\kappa\,\sqrt{q}$ with $\kappa$ a root of unity.

Let $\alpha_1,\alpha_2$ be the eigenvalues of 
$\Fr_q^{\ast}$ on 
$H_{\mathrm{c}}^1(D_{c,d})[\nu_0\boxtimes{}^0 \psi]$.  
Then $\alpha_i=\kappa_i\sqrt{q}$ with 
$\kappa_i$ a root of unity, and therefore 
$v_{\mathfrak{p}}(\alpha_i)=f/2$.
Since
\[
S(a,b,\psi)=-\alpha_1-\alpha_2,
\]
we obtain
\[
v_{\mathfrak{p}}(S(a,b,\psi))
=\min\{v_{\mathfrak{p}}(\alpha_1),
      v_{\mathfrak{p}}(\alpha_2)\}
\ge f/2.
\]
Hence we have proved the claim. 
\end{proof}
\begin{proposition}\label{ve}
Let $\xi \in W_{\mathrm{prim}}^{\vee}$. 
Then we have $v_{\mathfrak{p}}(S_{\xi}\, U_{\xi})=f/2$.
\end{proposition}

\begin{proof}
We write $\nu$ for $\nu_q$. 
By $p=3$, we have $\xi^2(x,0)=\xi(-x,x^3)$. 
Thus 
\[
S_{\xi}\, U_{\xi}
= \sum_{\substack{x\neq 0 \\ y\in\F_q}}
\nu(x)\,\xi(-x,x^3)\,\xi(y,0)
= \sum_{\substack{x\neq 0 \\ y}}
\nu(x)\,\xi(-x+y,\, x^3 + xy(-x+y)).
\]

Put $z=-x+y$.  
Then $y=x+z$, and we obtain
\[
S_{\xi}\, U_{\xi}
= \sum_{x\neq 0}\sum_{z}
\nu(x)\,\xi\!\left(z,\, x^3 + x(x+z)z\right).
\]

We use the notation $\psi={}^0 \psi\circ\Tr_{q/3}$ with 
${}^0 \psi \in \F_3^{\vee}\setminus\{1\}$, so that
$\xi(x_0,x_1)=\xi(x_0,0) \psi(ax_1)$ with $a \in \F_q^{\times}$. 
We get
\begin{align*}
S_{\xi}\, U_{\xi}
= \sum_{z}\xi(z,0)
 \sum_{x\neq 0}
 \nu(x)\psi\!\left(a x^3 + a x(x+z)z\right). 
\end{align*}
In the notation of Lemma \ref{vv}, 
\[
S_{\xi}\, U_{\xi}
=\sum_{z} \xi(z,0) S(a,z,\psi). 
\]
For $x,y \in \mathbb{Q}_3(\zeta_9)$ and $r \in \mathbb{Q}$, 
we write 
$x \equiv y \pmod{r+}$ if $v_{\mathfrak{p}}(x-y)>r$. 
By Lemma \ref{vv} and $v_{\mathfrak{p}}(\xi(z,0)-1)>0$ for any $z \in \F_q$, we have 
\begin{equation}\label{cong}
S_{\xi}\, U_{\xi}
=\sum_{z} \xi(z,0) S(a,z,\psi)
\equiv \sum_z S(a,z,\psi) \pmod{f/2+}. 
\end{equation}

Completing the square in $z$ and using 
$p=3$, we have
\[
a x^3 + ax(x+z)z
= a x (z - x)^2.
\]

Hence
\[
 \sum_z S(a,z,\psi) 
= \sum_{x\neq 0} \nu(x)
   \sum_{z} \psi\!\left( a x (z-x)^2 \right).
\]

For fixed $x\neq 0$, the inner sum is a quadratic Gauss sum:
\[
\sum_{z} \psi\!\left(a x (z-x)^2\right)
= \nu(ax) G(\psi) \quad \textrm{for $a \in \F_q^{\times}$}.
\]

Therefore,
\[
 \sum_z S(a,z,\psi) 
= \sum_{x\neq 0} \nu(x)\nu(ax)\,G(\psi)
= \nu(a) \sum_{x\neq 0} G(\psi)
= \nu(a)(q-1)G(\psi).
\]

Finally, since $v_{\mathfrak{p}}(G(\psi))=f/2$ and 
\eqref{cong}, we obtain
\[
v_{\mathfrak{p}}(S_{\xi}\, U_{\xi})
= v_{\mathfrak{p}}(\nu(a)(q-1) G(\psi))= f/2.
\]
We obtain the claim. 
\end{proof}
We show one of our main theorems. 
\begin{corollary}\label{sl}
Let $\xi \in W_{\mathrm{prim}}^{\vee}$. 
Then we have $v_{\mathfrak{p}}(S_{\xi})=f/3$
and $v_{\mathfrak{p}}(U_{\xi})=f/6$. 
\end{corollary}
\begin{proof}

By Lemma \ref{2}, 
\begin{equation}\label{sxi}
S_{\xi}^2=q-\nu_q(2) \,G(\psi)\, U_{\xi}.  
\end{equation}
If $v_{\mathfrak{p}}(\nu_q(2) \,G(\psi)\, U_{\xi}) \geq f$, we have 
$v_{\mathfrak{p}}(S_{\xi}) \geq f/2$ and $v_{\mathfrak{p}}(U_{\xi}) \geq f/2$
by $v_{\mathfrak{p}}(G(\psi))=f/2$. 
By Proposition \ref{ve}, we have $v_{\mathfrak{p}}(U_{\xi})=(f/2)-v_{\mathfrak{p}}(S_{\xi}) \leq 0$, which 
contradicts $v_{\mathfrak{p}}(U_{\xi}) \geq f/2$. Hence 
$v_{\mathfrak{p}}(\nu_q(2) \,G(\psi)\, U_{\xi}) < f$ and again by \eqref{sxi}, 
\[
v_{\mathfrak{p}}(S_{\xi}^2)=v_{\mathfrak{p}}(\nu_q(2) \,G(\psi)\, U_{\xi})=(f/2)+v_{\mathfrak{p}}(U_{\xi})=f-v_{\mathfrak{p}}(S_{\xi}), 
\]
where we use Proposition \ref{ve} at the last equality. 
Thus the claim follows. 
\end{proof}
\begin{corollary}
Let $\xi \in W_{\mathrm{prim}}^{\vee}$ and let 
\[
T_{\xi}=\sum_{x \in \F_q} \xi(x^2,0), \quad 
V_{\xi,c}:=
\sum_{x\in\F_q^{\times}}
\xi\!\left(x,(2c-1)x-c^2 x^{-1}\right)\nu_q(x) \quad 
\textrm{for $c \in \F_q^{\times}$}. 
\]
Then we have $v_{\mathfrak{p}}(T_{\xi})=
v_{\mathfrak{p}}(V_{\xi,c})=f/6$
for any $c \in \F_q^{\times}.$
\end{corollary}
\begin{proof}
The assertion for $T_{\xi}$
follows from Proposition \ref{3} and 
Corollary \ref{sl}. 
The assertion for $V_{\xi,c}$ follows from Lemma \ref{prod} and 
Corollary \ref{sl}. 
\end{proof} 
For $\alpha, \beta \in \overline{\Q}_{\ell}^{\times}$, 
if $\alpha/\beta$ is a root of unity, 
we write $\alpha \equiv \beta \pmod{\mu_{\infty}(\overline{\Q}_{\ell})}$. 
The following lemma is shown in a more general setting 
by a different method in Corollary \ref{kkatz}. 
\begin{lemma}\label{laumon}
Let $\xi \in W_{\mathrm{prim}}^{\vee}$. 
Let $\nu_0$ be a nontrivial character of $\{\pm 1\}$
and let $\mathscr{K}_{\nu_0}(x)$ be the 
Kummer sheaf on $\mathbb{G}_{\mathrm{m},\, \F_q}$
defined by $y^2=x$ and $\nu_0$. 
Then we have 
\[
\det(-\mathrm{Fr}_q^\ast; H_{\rm c}^1(\mathbb{G}_{\mathrm{m}},\mathscr{K}_{\nu_0}(x) \otimes 
\mathscr{L}_{\xi}(x,0))) \equiv q\, G(\psi) 
\pmod{\mu_{\infty}(\overline{\mathbb{Q}}_{\ell})}. 
\]
\end{lemma}
\begin{proof}
We simply write $\mathscr{L}:=\mathscr{K}_{\nu_0}(x) \otimes 
\mathscr{L}_{\xi}(x,0)$. 
Let $D$ be the projective line over $\F_q$ 
and we regard $\mathbb{G}_{\mathrm{m},\, \F_q}=\Spec \F_q[x^{\pm 1}]$ as the open subscheme of $D$. 
Let $j \colon \mathbb{G}_{\mathrm{m},\, \F_q}
\to D$ be the open immersion. 
%For a closed point $x$ of $D$, let $D_{(x)}$ denote the henselization of $D$ at $x$ and let $\eta_{(x)}$ denote the generic point of $D_{(x)}$. 
We freely use the notation in \cite{La}. 
We recall $H_{\rm c}^i(\mathbb{G}_{\mathrm{m}},\mathscr{L})=0$ for
$i \neq 1$. 
By the product formula \cite[Th\'eor\`eme (3.2.1.1)]{La} for $X=D$, $K=j_! \mathscr{L}$, 
we have 
\begin{align*}
\det(-\mathrm{Fr}_q^\ast; H_{\rm c}^1(\mathbb{G}_{\mathrm{m}},\mathscr{L}))
=q\, \epsilon(D_{(0)}, \mathscr{L}|_{\eta_0},d\pi_0)\, 
\epsilon\!\left(D_{(\infty)}, \mathscr{L}|_{\eta_{\infty}},-\frac{d\pi_{\infty}}{\pi_{\infty}^2}\right), 
\end{align*}
where $\pi_0=x$ and $\pi_{\infty}=x^{-1}$ (cf.\ 
\cite[(3.1.1.5)]{La}). 
By \cite[(3.1.5.6) and (3.5.3.1)]{La}, we have 
\[
\epsilon(D_{(0)}, \mathscr{L}|_{\eta_0},d\pi_0)
\equiv \epsilon(D_{(0)}, \mathscr{K}_{\nu_0}(x)|_{\eta_0},d\pi_0) \equiv G(\psi) 
\pmod{\mu_{\infty}(\overline{\mathbb{Q}}_{\ell})}. 
\]
We recall that the Swan conductor 
exponent of $\mathscr{L}$ is $3$ by Lemma \ref{lemma:Brylinski}.  
We also have 
\begin{align*}
\epsilon\!\left(D_{(\infty)}, \mathscr{L}|_{\eta_{\infty}},-\frac{d\pi_{\infty}}{\pi_{\infty}^2}\right) & \equiv q^{-2} \, 
\epsilon(D_{(\infty)}, \mathscr{L}|_{\eta_{\infty}},d\pi_{\infty}) \pmod{\mu_{\infty}(\overline{\mathbb{Q}}_{\ell})}\\
&=q^{-2}\, \epsilon(\chi_{\mathscr{L}}, \Psi_{d\pi_{\infty}}) \\
& =q^{-2}\int_{\pi_{\infty}^{-4} \mathcal{O}_{\infty}^{\times}}
\chi_{\mathscr{L}}^{-1}(z) \Psi_{d\pi_{\infty}}(z)dz, 
\end{align*}
where we use \cite[(3.1.5.5)]{La}, \cite[Th\'eor\`eme (3.1.5.4)(v)]{La} and \cite[(3.1.3.2)]{La}, respectively. Here 
$\chi_{\mathscr{L}}$ is a character of 
$\F_q((\pi_{\infty}))^{\times}$ induced by 
$\mathscr{L}|_{\eta_{\infty}}$ via the local class field theory, $\mathcal{O}_{\infty}:=
\F_q[[\pi_{\infty}]]$, and 
$\int_{\mathcal{O}_{\infty}} dz=1$. 
Applying \cite[Proposition~8.7(ii)]{AS} to the local field
$K=\F_q((\pi_{\infty}))$ with $\chi=\chi_{\mathscr{L}}$ and
$\psi=\Psi_{d\pi_{\infty}}$, we note that by Lemma~\ref{lemma:Brylinski}
the character $\chi_{\mathscr{L}}$ has Swan conductor exponent $3$.
In the notation of loc.\ cit.\ we therefore take
$\mathrm{ord}(\psi)=0$, $\mathrm{ord}(\beta)=1$ and $\mathrm{ord}(c)=-4$,
so that the relevant integral is over $\pi_{\infty}^{-4}\mathcal{O}_{\infty}^{\times}$.
With the normalization $\int_{\mathcal{O}_{\infty}}dz=1$, Proposition~8.7(ii)
gives
\[
\epsilon(\chi_{\mathscr{L}},\Psi_{d\pi_{\infty}})
\equiv q^2 \pmod{\mu_{\infty}(\overline{\Q}_{\ell})}.
\]

\end{proof}

Now we prove one of our main theorems.

\begin{theorem}
We simply write $D$ for $C_{2,2,3,q}$ and assume $q=3^f$.
\begin{itemize}
\item[{\rm (1)}]
Let $\nu_0$ be a nontrivial character of $\{\pm 1\}$. 
Then
\[
H^1(\overline{D}) \simeq 
\bigoplus_{\xi \in W^{\vee}}
H_{\rm c}^1(\mathbb{A}^1,\mathscr{L}_{\xi}(x,0))
\oplus H_{\rm c}^1(\mathbb{G}_{\mathrm{m}},\mathscr{K}_{\nu_0}(x) \otimes \mathscr{L}_{\psi}(x,0)). 
\]
Moreover, $g(\overline{D})=(q-1)(5q+1)/2$. 

\item[{\rm (2)}] 
The multiset of slopes of 
$H_{\rm c}^1(\mathbb{G}_{\mathrm{m}},\mathscr{K}_{\nu_0}(x) \otimes \mathscr{L}_{\psi}(x,0))$ with respect to
$\mathrm{Fr}_q^\ast$ is  
\[
\begin{cases}
\left\{\dfrac{1}{2}\right\} & \text{if $\xi \not\in W_{\mathrm{prim}}^{\vee}$ and $\xi \neq 1$}, \\[0.3cm]
\left\{\dfrac{1}{2},\ \dfrac{1}{6},\ \dfrac{5}{6}\right\} & \text{if $\xi \in W_{\mathrm{prim}}^{\vee}$}. 
\end{cases}
\]

\item[{\rm (3)}]
The multiset of slopes of $\overline{D}$ with respect to
$\mathrm{Fr}_q^\ast$ is  
\[
\left\{\frac{1}{2},\ \frac{1}{3},\ \frac{2}{3},\ 
\frac{1}{6},\ \frac{5}{6}\right\}.
\]
\end{itemize}
\end{theorem}

\begin{proof}
We prove (1). 
By the definition of $D$ and the proof of Corollary \ref{genus}, we have 
\[
H^1(\overline{D}) \simeq H_{\rm c}^1(D) \simeq \bigoplus_{\xi \in W^{\vee}}
H_{\rm c}^1(\mathbb{A}^1,\mathscr{L}_{\xi}(x^2,0)).  
\]

For $\mathscr{F} \in \{\mathscr{L}_{\xi}(x,0),\ \mathscr{L}_{\xi}(x^2,0)\}$, 
we have a short exact sequence 
\[
0 \to \overline{\mathbb{Q}}_{\ell}
\to H_{\rm c}^1(\mathbb{G}_{\mathrm{m}},\mathscr{F}) 
\to H_{\rm c}^1(\mathbb{A}^1,\mathscr{F}) \to 0.
\]
We simply write $\mathscr{L}_{\nu_0,\xi}$
for $\mathscr{K}_{\nu_0}(x) \otimes \mathscr{L}_{\xi}(x,0)$. 
By the projection formula, 
\[
H_{\rm c}^1(\mathbb{G}_{\mathrm{m}},\mathscr{L}_{\xi}(x^2,0)) \simeq 
H_{\rm c}^1(\mathbb{G}_{\mathrm{m}},\mathscr{L}_{\xi}(x,0)) \oplus 
H_{\rm c}^1(\mathbb{G}_{\mathrm{m}},\mathscr{L}_{\nu_0,\xi}). 
\]

The above short exact sequence then implies
\[
H_{\rm c}^1(\mathbb{A}^1,\mathscr{L}_{\xi}(x^2,0))
\simeq 
H_{\rm c}^1(\mathbb{A}^1,\mathscr{L}_{\xi}(x,0)) \oplus 
H_{\rm c}^1(\mathbb{G}_{\mathrm{m}},\mathscr{L}_{\nu_0,\xi}). 
\]

Thus the first claim follows. 
The genus formula follows from Lemma \ref{genus} for $l=p=3$ and $n=2$.

We now prove (2). 
Assume $\xi \not\in W_{\rm prim}^{\vee}$. 
If $\xi=1$, then
\[
H_{\rm c}^1(\mathbb{G}_{\mathrm{m}},\mathscr{L}_{\nu_0,\xi}) 
\simeq 
H_{\rm c}^1(\mathbb{G}_{\mathrm{m}},\mathscr{K}_{\nu_0}(x))=0.
\]
Thus assume $\xi\neq 1$. 
It is well known that 
the Frobenius slope of 
$H_{\rm c}^1(\mathbb{G}_{\mathrm{m}},\mathscr{L}_{\nu_0,\xi})$ 
with respect to $\mathrm{Fr}_q^\ast$ is $1/2$, since 
$\mathscr{L}_{\xi}(x,0)\simeq \mathscr{L}_{\psi}(x)$ for some 
$\psi \in \F_q^{\vee}\setminus \{1\}$.

Assume now that $\xi \in W_{\rm prim}^{\vee}$. 
Since $\mathrm{Swan}_{\infty}(\mathscr{L}_{\nu_0,\xi})=3$ by Lemma \ref{lemma:Brylinski}, 
\[
\dim H_{\rm c}^1(\mathbb{G}_{\mathrm{m}},\mathscr{L}_{\nu_0,\xi})=3. 
\]

Lemma \ref{laumon} yields
\begin{equation}\label{v3} 
v_{\mathfrak{p}}\left(\det(-\mathrm{Fr}_q^\ast; H_{\rm c}^1(\mathbb{G}_{\mathrm{m}},\mathscr{L}_{\nu_0,\xi}))\right)=\frac{3f}{2}. 
\end{equation}

Let $\alpha,\,\beta,\,\gamma$ be the Frobenius eigenvalues of $\mathrm{Fr}_q^\ast$ 
on $H_{\rm c}^1(\mathbb{G}_{\mathrm{m}},\mathscr{L}_{\nu_0,\xi})$, ordered so that 
$a:=v_{\mathfrak{p}}(\alpha)\le b:=v_{\mathfrak{p}}(\beta)\le c:=v_{\mathfrak{p}}(\gamma)$. 
Since $v_{\mathfrak{p}}(U_{\xi})=f/6$ by Corollary \ref{sl}, 
we have $a\le v_{\mathfrak{p}}(\alpha+\beta+\gamma)=f/6$. 

By Poincar\'e duality and $\mathscr{L}_{\nu_0,\xi}^{\vee}
\simeq \mathscr{L}_{\nu_0,\xi^{-1}}$,
\begin{gather}\label{ppo}
\begin{aligned}
H_{\rm c}^1(\mathbb{G}_{\mathrm{m}}, \mathscr{L}_{\nu_0,\xi})^{\vee}(-1)
\simeq 
H^1(\mathbb{G}_{\mathrm{m}},\mathscr{L}_{\nu_0,\xi^{-1}}) \xleftarrow{\sim} 
H_{\rm c}^1(\mathbb{G}_{\mathrm{m}},\mathscr{L}_{\nu_0,\xi^{-1}}).
\end{aligned}
\end{gather}

Thus, since $v_{\mathfrak{p}}(U_{\xi^{-1}})=f/6$, 
\begin{equation}\label{poincare}
v_{\mathfrak{p}}(q\alpha^{-1}+q\beta^{-1}+q\gamma^{-1})=\frac{f}{6}.
\end{equation}

From \eqref{v3}, we have $a+b+c=3f/2$. 

Assume $a<f/6$. Then $a=b$, because $v_{\mathfrak{p}}(\alpha+\beta+\gamma)=f/6$. 
Hence both $v_{\mathfrak{p}}(q\alpha^{-1})$ and $v_{\mathfrak{p}}(q\beta^{-1})$ exceed $5f/6$. 
Then $c=(3f/2)-2a>5f/6$, giving 
$v_{\mathfrak{p}}(q\gamma^{-1})=f-c<f/6$, which contradicts \eqref{poincare}. 
Thus $a=f/6$. 

Replacing $\xi$ by $\xi^{-1}$ and using \eqref{ppo}, we obtain  
$\min\{v_{\mathfrak{p}}(q\alpha^{-1}),v_{\mathfrak{p}}(q\beta^{-1}),v_{\mathfrak{p}}(q\gamma^{-1})\}
= v_{\mathfrak{p}}(q\gamma^{-1})=f/6$, hence $c=5f/6$.  
Since $a+b+c=3f/2$, we get $b=f/2$.  

This proves the claim. 

Finally, (3) follows from (1) and (2).
\end{proof}

%For $\psi \in \F_q^{\vee}$, let $\mathscr{L}_{\psi}(x)$ denote the sheaf on $\mathbb{A}^1_{\F_q}$ defined by $y^q-y=x$ and $\psi$. For a morphism $f \colon X \to \mathbb{A}^1_{\F_q}$, let $\mathscr{L}_{\psi}(f)$ denote the pull-back of $\mathscr{L}_{\psi}(x)$ by $f$. 

\section{Covering of the curve}\label{s4}
We simply write $C_n$ for $C_{1,n,p,q}$.
Let $P(x) \in \F_q[x]$ be an $\F_p$-linearized separable polynomial.  
We define $C_{P,n}$ by the cartesian diagram
\[
\xymatrix{
C_{P,n} \ar[d]^{\mathrm{pr}'_1}\ar[rr]^{\varphi_n} && C_n \ar[d]^{\mathrm{pr}_1} \\
\mathbb{A}_{\F_q}^1 \ar[rr]^{x \mapsto P(x)} &&  \mathbb{A}^1_{\F_q}, 
}
\]
where $\mathrm{pr}_1$ is given by $(x_0,\ldots,x_{n-1}) \mapsto x_0$. 
Our main aim in this section is to determine the $L$-polynomial of 
$\overline{C}_{P,n}$
in the case $n=2$ and $p=3$.

We note that $C_{P,1}$ is isomorphic to 
$\mathbb{A}^1_{\F_q}$. 

We define an algebraic group 
$\mathbb{W}^P_{n,q}$ by the cartesian diagram
\[
\xymatrix{
\mathbb{W}^P_{n,q} \ar[rr]^{\phi_n}\ar[d]& & \mathbb{W}_{n,q} \ar[d]^{\mathrm{pr}_1} \\
\mathbb{G}_{\mathrm{a},\,\F_q} \ar[rr]^{x \mapsto P(x)} & & \mathbb{G}_{\mathrm{a},\,\F_q}, 
}
\]
where $\mathrm{pr}_1$ is given by $(x_0,\ldots,x_{n-1}) \mapsto x_0$. 
We consider 
\[
\phi_n \colon \mathbb{W}^P_{n,q}(\F) \to 
\mathbb{W}_{n,q}(\F)
\]
and we define 
\[
W_{P,n}:=\phi_n^{-1}(W_{n,q}). 
\]
The morphisms
$C_{P,n} \xrightarrow{\varphi_n} C_n \subset \mathbb{W}_{n,q}$
and 
$\mathrm{pr}'_1 \colon C_{P,n} \to \mathbb{G}_{\mathrm{a},\,\F_q}$
induce a morphism 
$C_{P,n} \hookrightarrow \mathbb{W}_{n,q}^P$. 
By $\mathbb{W}_{n,q}^P =\mathbb{A}^1_{\F_q} \times_{P,\, \mathbb{A}^1_{\F_q}} \mathbb{W}_{n,q} $, 
we have 
\begin{align*}
\mathbb{W}_{n,q}^P \times_{\phi_n,\, \mathbb{W}_{n,q}} C_n
\simeq 
(\mathbb{A}^1_{\F_q} \times_{P,\, \mathbb{A}^1_{\F_q}} \mathbb{W}_{n,q}) \times_{\mathbb{W}_{n,q}} C_n
\simeq \mathbb{A}^1_{\F_q}\times_{P,\, \mathbb{A}^1_{\F_q}}
C_n=C_{P,n}. 
\end{align*}
Hence, 
for $x \in \mathbb{W}_{n,q, \F}^P$, 
\[
x \in  C_{P,n,\F} \iff \phi_n(x) \in C_{P,\F}. 
\] 
For $x \in C_{P,n,\F}$ and 
$b \in W_{P,n}$, we have $x+b \in C_{P,n,\F}$ 
since $\phi_n(x+b)=\phi_n(x)+\phi_n(b) 
\in C_{P,\F}$ by $\phi_n(b) \in W_{n,q}$. 
\begin{lemma}\label{product}
Let 
$K_P:=\{x \in \F \mid P(x)=0\}$. 
The group $W_{P, n}$ is isomorphic to
a product group 
$K_P \times W_{n,q}$. 
\end{lemma}
\begin{proof}
We have a short exact sequence 
\[
0 \to K_P \to W_{P,n} \xrightarrow{\phi_n}
W_{n,q} \to 0. 
\]
Let $x=(x_0,\ldots,x_{n-1}) \in W_{P,n} \subset
\F^n$. Then $\phi_n(x)=
(P(x_0),x_1,\ldots,x_{n-1}) \in W_{n,q}$.  
Let $x \in W_{P,n}$. We show $p^n x=0$. 
We have $\phi_n(p^n x)=p^n \phi_n(x)=0$ by Lemma \ref{cyclic}. 
Hence $p^n x \in \Ker \phi_n$. 
Then we can write $(y_0, 0'\mathrm{s})=p^n x$. 
By $p^n x=(0,\ast,\ldots,\ast)$, we know $y_0=0$. 
Hence $p^nx=0$. Hence the short exact sequence 
is split.  
\end{proof}

We have a finite \'etale Galois covering 
\begin{align*}
\pi_n & \colon C_n \to \mathbb{A}^1_{\F_q}, \quad 
(x_0,\ldots,x_{n-1}) \mapsto x_0^q-x_0, 
\end{align*}
whose Galois group is $W_{n,q}=\mathbb{W}_n(\F_q)$. 

We take $\psi \in \F_q^{\vee} \setminus \{1\}$. 
Let $r$ be the minimal positive integer such that  
$K_P \subset \F_{q^r}$. 
Let $Q(x) \in \F_q[x]$ be an $\F_p$-linearized  polynomial such that $Q(P(x))=P(Q(x))=x^{q^r}-x$. 
We have a surjection 
\[
\F_{q^r}\twoheadrightarrow K_P, \quad x \mapsto Q(x). 
\]
Taking the dual, 
we have 
\[
K_P^{\vee} \hookrightarrow \F_{q^r}^{\vee}. 
\]
We have the bijection 
\[
\Psi \colon \F_{q^r} \to \F_{q^r}^{\vee}, \quad 
a \mapsto  (x \mapsto \psi\circ \Tr_{q^r/q}(ax)). 
\]
We define 
$K_{P,\psi}^0 \subset \F_{q^r}$ to be 
$\Psi^{-1}(K_P^{\vee})$.
We take a minimal extension 
$\F_{q^s}/\F_{q^r}$ such that, for any $a \in K_{P,\psi}^0$, 
there exist solutions $x \in \F_{q^s}$ such that $x^q-x=a$. 
We consider a map 
$g \colon \F_{q^s} \to \F_{q^s}, \ x \mapsto 
x^{1/q}-x$ and 
 choose a subset $K_{P,\psi}\subset\F_{q^s}$ such that
\[
g\big|_{K_{P,\psi}}\colon K_{P,\psi}\xrightarrow{\sim} K_{P,\psi}^0.
\]
The morphism 
\[
\pi_n \circ \varphi_n \colon C_{P,n,\F_{q^s}} \to \mathbb{A}^1_{\F_{q^s}}, \quad 
(x_0,\ldots,x_{n-1}) \mapsto P(x_0)^q-P(x_0)
\]
is a finite \'etale Galois covering whose Galois group is $W_{P,n}$. 
\begin{proposition}\label{pp}
We have 
\[
\pi_{n\ast}\varphi_{n\ast} \overline{\mathbb{Q}}_{\ell}
\simeq 
\bigoplus_{\xi \in W_{n,q}^{\vee},\, 
\alpha \in K_{P,\psi}} \mathscr{L}_{\psi}(\alpha x) 
\otimes \mathscr{L}_{\xi}(x,0'\mathrm{s}). 
\]
\end{proposition}
\begin{proof}
We consider the commutative diagram
\[
\xymatrix{
C_{P,n,\F} \ar[rr]^{\varphi_n}\ar[d]^{\mathrm{pr}'_1} & & C_{n,\F} \ar[r]^-{\pi_n}\ar[d]^{\mathrm{pr}_1} & \mathbb{A}^1 \\
\mathbb{A}^1 \ar[rr]^{x \mapsto P(x)} && \mathbb{A}^1, \ar[ur]_{f_0} 
}
\]
where $f_0$ is given by $x \mapsto x^q-x$. 
Here, the left commutative diagram is cartesian. 
The morphism
$\mathbb{A}^1 \to \mathbb{A}^1,\ x \mapsto x^{q^r}-x$ factors 
through 
\[
\mathbb{A}^1 \xrightarrow{x \mapsto Q(x)} \mathbb{A}^1 
\xrightarrow{x \mapsto P(x)} \mathbb{A}^1. 
\]
The group $K_P$ is regarded as a quotient of $\F_{q^r}$ by 
$x \mapsto Q(x)$. 
For any $\psi' \in K_P^{\vee} \subset \F_{q^r}^{\vee}$, 
we have $P^{\ast}\mathscr{L}_{\psi'}(x) \simeq 
\overline{\Q}_{\ell}$. Thus we obtain a nonzero morphism $\mathscr{L}_{\psi'}(x)
\to P_\ast \overline{\Q}_{\ell}$. Thus 
we obtain an injection $\bigoplus_{\psi' \in K_P^{\vee}} 
\mathscr{L}_{\psi'}(x) \to P_{\ast} \overline{\mathbb{Q}}_{\ell}$. By comparing the ranks, we obtain an isomorphism 
$\bigoplus_{\psi' \in K_P^{\vee}} 
\mathscr{L}_{\psi'}(x) \simeq P_{\ast} \overline{\mathbb{Q}}_{\ell}$. 
By the proper base change theorem, 
\[
\varphi_{n\ast} \overline{\mathbb{Q}}_{\ell}\simeq 
\mathrm{pr}_1^\ast P_{\ast} \overline{\Q}_{\ell} 
\simeq 
\bigoplus_{\psi' \in K_P^{\vee}} 
\mathrm{pr}_1^\ast\mathscr{L}_{\psi'}(x) \simeq 
\bigoplus_{a \in K_{P,\psi}^0} 
\mathrm{pr}_1^\ast\mathscr{L}_{\psi}(a x). 
\]
Let $H \subset W_{n,q}$ denote the subgroup consisting of the elements $(0,a_1,\ldots,a_{n-1})$. 
Let $\mathscr{L}_{H,\psi'}$ denote the 
sheaf defined by the $H$-covering $\varphi_n$
and $\psi' \in H^{\vee}$. 
Applying $\pi_{n\ast}$ and using the projection formula yield 
\begin{align*}
\pi_{n\ast}\varphi_{n\ast} \overline{\mathbb{Q}}_{\ell}
& \simeq \bigoplus_{a \in K_{P,\psi}^0} 
f_{0\ast} \mathrm{pr}_{1\ast}\mathrm{pr}_1^\ast\mathscr{L}_{\psi}(a x) \\
& \simeq 
\bigoplus_{a \in K_{P,\psi}^0} 
f_{0\ast} (\mathscr{L}_{\psi}(a x)\otimes 
\mathrm{pr}_{1\ast} \overline{\mathbb{Q}}_{\ell}) \\
& \simeq 
\bigoplus_{a \in K_{P,\psi}^0} \bigoplus_{\psi' \in H^{\vee}}
f_{0\ast} (\mathscr{L}_{\psi}(a x)\otimes 
\mathscr{L}_{H,\psi'}). 
\end{align*}
We take $\alpha \in K_{P,\psi}$ such that 
$\alpha^{1/q}-\alpha=a$. 
Then $\mathscr{L}_{\psi}(a x)
\simeq  \mathscr{L}_{\psi}(\alpha(x^q-x)) \simeq 
f_0^\ast \mathscr{L}_{\psi}(\alpha x)$. 
Thus 
\begin{align*}
f_{0\ast} (\mathscr{L}_{\psi}(a x)\otimes 
\mathscr{L}_{H,\psi'})
& \simeq f_{0\ast} (f_0^\ast \mathscr{L}_{\psi}(\alpha x)\otimes 
\mathscr{L}_{H,\psi'}) \\
& \simeq 
\mathscr{L}_{\psi}(\alpha x)\otimes f_{0\ast} 
\mathscr{L}_{H,\psi'} \\
& \simeq 
\bigoplus_{\xi \in W_{\psi'}^{\vee}} 
\mathscr{L}_{\psi}(\alpha x)\otimes  
\mathscr{L}_{\xi}(x,0'\mathrm{s}), 
\end{align*}
where $W_{\psi'}^{\vee}:=\{\xi \in W_{n,q}^{\vee} \mid \xi|_H=\psi'\}$. 

Hence we obtain 
\[
\pi_{n\ast}\varphi_{n\ast} \overline{\mathbb{Q}}_{\ell}
\simeq \bigoplus_{\alpha \in K_{P,\psi}}
\bigoplus_{\xi \in W_{n,q}^{\vee}} 
\mathscr{L}_{\psi}(\alpha x) \otimes \mathscr{L}_{\xi}(x,0'\mathrm{s}). 
\]
\end{proof}
The morphism $C_n \to C_{n-1}, \quad (x_0,\ldots,x_{n-1})
\mapsto (x_0,\ldots,x_{n-2})$ indues a natural projection 
\[
r_n \colon C_{P,n} \to C_{P,n-1}. 
\]
Clearly we have $\pi_n \circ \varphi_n=\pi_{n-1} \circ \varphi_{n-1} \circ r_n$. 
The morphism $r_n$ induces 
\[
r_n^\ast \colon H_{\rm c}^1(C_{P,n-1}) \hookrightarrow 
H_{\rm c}^1(C_{P,n}). 
\]
We have a surjective homomorphism
$W_{n,q} \to W_{n-1,q},\ (x_0,\ldots,x_{n-1}) \to (x_0,\ldots,x_{n-2})$, which 
induces an injection $W_{n-1,q}^{\vee} \to W_{n,q}^{\vee}$. 
\begin{lemma}\label{rn}
Under the identifications in Proposition~\ref{pp}, the subspace 
\[
r_n^\ast \colon 
H_{\rm c}^1(C_{P,n-1}) \hookrightarrow H_{\rm c}^1(C_{P,n})
\]
corresponds to
\[
\bigoplus_{\xi \in W_{n-1,q}^{\vee},\, \alpha \in K_{P,\psi}} 
H_{\rm c}^1(\mathbb{A}^1, \mathscr{L}_{\psi}(\alpha x) \otimes \mathscr{L}_{\xi}(x,0'\mathrm{s})),
\]
which is exactly the morphism induced by the injection $W_{n-1,q}^{\vee} \hookrightarrow W_{n,q}^{\vee}$. 
\end{lemma}
\begin{proof}
Let $\F_q \hookrightarrow W_{n,q},\ t \mapsto (0'\mathrm{s},t)$. 
We have an isomorphism 
$C_{P,n}/\F_q \simeq C_{P,n-1}$. 
Thus 
\[
H_{\rm c}^1(C_{P,n-1})=H_{\rm c}^1(C_{P,n})^{\F_q}, 
\]
where $(-)^{\F_q}$ denotes the $\F_q$-fixed part. 
We have $\xi|_{\F_q}=1 \iff \xi \in W_{n-1,q}^{\vee}$. Hence the claim follows. 
\end{proof}
\begin{corollary}\label{rnc}
We define 
\[
V_n:=H_{\rm c}^1(C_{P,n})/H_{\rm c}^1(C_{P,n-1}). 
\]
Then we have an isomorphism 
\[
V_n \simeq \bigoplus_{\xi \in W_{n,q,\mathrm{prim}}^{\vee}} \bigoplus_{\alpha \in K_{P,\psi}} 
H_{\rm c}^1(\mathbb{A}^1,\mathscr{L}_{\psi}(\alpha x) \otimes \mathscr{L}_{\xi}(x,0'\mathrm{s})). 
\]
\end{corollary}
\begin{proof}
The claim follows from $W_{n,q,\mathrm{prim}}^{\vee}=W_{n,q}^{\vee}
\setminus W_{n-1,q}^{\vee}$ and Lemma \ref{rn}. 
\end{proof}

We recall that $W_{n,q,\mathrm{prim}}^{\vee}$
denotes the set of characters $\xi$ 
of $W_{n,q}$ such that $\xi|_{\{0'\mathrm{s}\} \times \F_q}\neq 1$.  
\begin{lemma}\label{chd}
Let ${}^0 \xi \in W_{n,p}^{\vee}$
be a faithful character and let 
$\xi \in W_{n,q,\rm{prim}}^{\vee}$. 
Choose $a=(a_0,\ldots,a_{n-1}) \in W_{n,q}$
such that $\xi=\xi_a$
in the notation of Lemma \ref{ch}.  
 Let ${}^0 \psi:={}^0 \xi|_{\{0'\mathrm{s}\} \times \F_p}$ and $\psi:= {}^0 \psi \circ \Tr_{q/p}$. 
 Let $a':=a+(0'\mathrm{s},\alpha^{p^{n-1}})$. 
 Then $\xi_{a'} \in W_{n,q,\mathrm{prim}}^{\vee}$ and 
 \[
 \mathscr{L}_{\psi}(\alpha x) \otimes \mathscr{L}_{\xi}(x,0'\mathrm{s})
 \simeq \mathscr{L}_{{}^0 \xi}(a' \cdot (x,0'\mathrm{s})). 
 \]
\end{lemma}
\begin{proof}
We recall
\[
a \cdot (x,0'\mathrm{s})=
(a_0x,a_1x^p,\ldots,a_{n-1} x^{p^{n-1}}). 
\]

We have 
\begin{align*}
\mathscr{L}_{\psi}(\alpha x) \otimes \mathscr{L}_{\xi}(x,0'\mathrm{s}) &\simeq 
\mathscr{L}_{{}^0 \xi}(0'\mathrm{s},(\alpha x)^{p^{n-1}}) \otimes 
\mathscr{L}_{{}^0 \xi}(a \cdot (x,0'\mathrm{s})) \\
&\simeq 
\mathscr{L}_{{}^0 \xi}((0'\mathrm{s},\alpha^{p^{n-1}}) \cdot  (x,0'\mathrm{s})) \otimes 
\mathscr{L}_{{}^0 \xi}(a \cdot (x,0'\mathrm{s}))\\
&\simeq \mathscr{L}_{{}^0 \xi}(a' \cdot (x,0'\mathrm{s})). 
\end{align*}
We write $a'=(a'_0,\ldots,a'_{n-1})$.  
By $a'_0=a_0 \neq 0$, we have $\xi_{a'} \in W_{n,q,\mathrm{prim}}^{\vee}$
by Lemma \ref{ch}. 
\end{proof}

\begin{corollary}\label{c42}
Let $V_n$ be as in Corollary \ref{rnc}. 
\begin{itemize}
\item[{\rm (1)}]
We have $\dim V_n= q^{n-1}(p^{n-1}-1)(q-1)\deg P$.
\item[{\rm (2)}] The curve $\overline{C}_{n,P}$ is geometrically connected. 
\item[{\rm (3)}] 
The natural map 
$H_{\rm c}^1(C_{P,n}) \to H^1(\overline{C}_{P,n})$
is an isomorphism.
We have 
\[
g(\overline{C}_{P,n})=\deg P \cdot \frac{q-1}{2} \sum_{i=1}^n q^{i-1}(p^{i-1}-1)
=\deg P \cdot g(\overline{C}_n). 
\]
\end{itemize}  
\end{corollary}
\begin{proof}
Let $\xi \in W_{n,q,\mathrm{prim}}^{\vee}$ and $\alpha \in K_{P,\psi}$. 
We simply write $\mathscr{L}_{\xi,\alpha}$ for 
$\mathscr{L}_{\psi}(\alpha x) \otimes \mathscr{L}_{\xi}(x,0'\mathrm{s})$. 
We show (1). 
The claim follows from Corollary \ref{rnc} and 
\[
\dim H_{\rm c}^1(\mathbb{A}^1,\mathscr{L}_{\xi,\alpha})=
p^{n-1}-1, \quad \# W_{n,q,\mathrm{prim}}^{\vee}=q^{n-1}(q-1), \quad 
\# K_{P,\psi}=\deg P,
\]
where the first equality follows from Lemmas \ref{lemma:Brylinski} and 
\ref{chd}. 

We show (2). 
By Poincar\'e duality, it suffices to show that $\dim H_{\rm c}^2(C_{P,n})=1$. 
We show this by induction on $n$.  
If $\xi\in W_{n,q,\mathrm{prim}}^{\vee}$, 
then $\mathscr{L}_{\xi,\alpha}$ is ramified, and 
$H_{\rm c}^2(\mathbb{A}^1,\mathscr{L}_{\xi,\alpha}) = 0$.
Assume $n=2$. 
By Proposition \ref{pp}, 
\begin{align*}
H_{\rm c}^2(C_{P,2})
&\simeq \bigoplus_{\xi \in W_{2,q}^{\vee},\ \alpha \in K_{P,\psi}} H_{\rm c}^2(\mathbb{A}^1,\mathscr{L}_{\xi,\alpha}) \\
&\simeq \bigoplus_{\alpha \in K_{P,\psi},\, b \in \F_q} H_{\rm c}^2(\mathbb{A}^1,\mathscr{L}_{\psi}((\alpha+b) x))\simeq 
H_{\rm c}^2(\mathbb{A}^1,\overline{\mathbb{Q}}_{\ell})
\simeq \overline{\mathbb{Q}}_{\ell}(-1). 
\end{align*}
Assume $n \ge 3$. 
By (1) and 
$H_{\rm c}^2(\mathbb{A}^1, \mathscr{L}_{\xi,\alpha})=0$ for $\xi\in W_{n,q,\mathrm{prim}}^{\vee}$, 
we have an isomorphism $H_{\rm c}^2(C_{P,n-1}) \xrightarrow{\sim} H_{\rm c}^2(C_{P,n})$. Thus the claim 
follows from the induction hypothesis. 

We show (3). 
By induction on $n$, 
we show that 
the forgetful map 
\[
H_{\rm c}^1(C_{P,n}) \to H^1(C_{P,n}) 
\]
 is an isomorphism.

The forgetful map
\[
H_{\rm c}^1(\mathbb{A}^1,\mathscr{L}_{\xi,\alpha}) \to 
H^1(\mathbb{A}^1,\mathscr{L}_{\xi,\alpha})
\]
is an isomorphism for $a \in K_{P,\psi}$
and $\xi \in W_{n,q,\mathrm{prim}}^{\vee}$. 
The claim for $n=2$ follows from this. 
By the induction hypothesis, the map 
\[
H_{\rm c}^1(C_{P,n}) \to H^1(C_{P,n}) 
\]
is an isomorphism. 

Hence we obtain an isomorphism $H_{\rm c}^1(C_{P,n}) 
\xrightarrow{\sim} H^1(\overline{C}_{P,n})$. 
The latter claim follows from (1) and 
\[
2g(\overline{C}_{P,n})=\dim H_{\rm c}^1(C_{P,n})=\sum_{i=1}^n \dim  V_i. 
\]
\end{proof}
\begin{corollary}\label{cla}
\begin{itemize}
\item[{\rm (1)}]
We define 
\[
L_{V_n/\F_{q^s}}(T):=\det(1-\mathrm{Fr}_{q^s}^\ast T; V_n). 
\]
Then we have 
\[
L_{\overline{C}_{P,n}/\F_{q^s}}(T)=
\prod_{i=2}^n L_{V_i/\F_{q^s}}(T)
\]
and 
\[
L_{V_n/\F_{q^s}}(T)=
\prod_{\xi \in W_{n,q,\rm prim}^{\vee}} \prod_{\alpha \in K_{P,\psi}}
\det(1-\mathrm{Fr}_{q^s}^\ast T; 
H_{\rm c}^1(\mathbb{A}^1,\mathscr{L}_{\psi}(\alpha x) \otimes \mathscr{L}_{\xi}(x,0))). 
\]
\item[{\rm (2)}] 
Assume that the set of 
the Frobenius slopes of $V_n$ with respect to $\mathrm{Fr}_q^\ast$ are 
\[
\left\{\frac{1}{p^{n-1}},\ \ldots,\ 
\frac{p^{n-1}-1}{p^{n-1}}\right\}
\]
for any $n$.  
Then the set of the 
Frobenius slopes of $\overline{C}_{P,n}$ with respect to 
$\mathrm{Fr}_q^\ast$ are 
\[
\left\{\frac{j}{p^i} \mid 1 \leq i \leq n-1,\ 1 \le j \leq p^i-1\right\}.
\]
\end{itemize}
\end{corollary}
\begin{proof}
The claim (1) is a consequence of Corollary  \ref{rnc}. 

The claim (2) follows from the assumption and Corollary \ref{c42}.  
\end{proof}
\begin{remark}\label{Cliu}
We consider Lemma \ref{chd}. 
We have 
\begin{equation}\label{important}
\mathscr{L}_{{}^0 \xi}(a' \cdot (x,0'\mathrm{s}))=\mathscr{L}_{{}^0 \xi}(a'_0x, \ldots,a'^{p^{-(n-1)}}_{n-1} x)
=\mathscr{L}_{{}^0 \xi}\left(\sum_{i=0}^{n-1} V^i\left([a'^{p^{-i}}_i x]\right)\right), 
\end{equation}
where $[-] \colon \mathbb{A}^1_{\F_q} \to \mathbb{W}_{n,q}$
denotes the Teichm\"uller lift and $V \colon \mathbb{W}_{n,q} \to 
\mathbb{W}_{n,q}$ is the Verschiebung (cf.\ \cite[the proof of Lemma 3.1]{Ka0a}). 
By $a'_0 \neq 0$ and \eqref{important}, the Newton polygon of the polynomial 
\[
\det(1-\mathrm{Fr}_{q^s}^\ast T; 
H_{\rm c}^1(\mathbb{A}^1,\mathscr{L}_{\psi}(\alpha x) \otimes \mathscr{L}_{\xi}(x,0)))
\]
is determined 
if \cite[Theorem 1.3]{Li} is correct.
If \cite[Theorem 1.3]{Li} holds 
for general $q$, the assumption in (2) in Corollary \ref{cla} is satisfied
and the Frobenius slopes of $\overline{C}_P$ 
are determined.

Although \cite[Theorem~1.3]{Li} is stated for arbitrary $q$, the arguments in the proof of \cite[Proposition~5.1]{Li} seem to be valid \emph{only} in the case $q=p$.
In particular, the precise Newton polygon information required in \cite[Theorem~1.3]{Li} cannot, in general, be deduced from \cite[Proposition~5.2]{Li}.
Indeed, when $q>p$, the valuations of the coefficients of the reciprocal characteristic polynomial of the matrix appearing in the last paragraph of \cite[p.~281]{Li} cannot, in general, be controlled using \cite[Proposition~5.2]{Li} alone.
\end{remark}
\begin{corollary}
Assume that $p$ is odd. 
The curve $\overline{C}_{P,n}$ is not supersingular. 
If $q=p$ and a smooth projective curve $X$ is a finite quotient of $\overline{C}_{P,n}$, the curve $X$ is not 
supersingular. 
\end{corollary}
\begin{proof}
The curve $\overline{C}_{P,n}$ admits a finite quotient $\overline{C}_n$. 
Since $\overline{C}_n$ is not supersingular by Corollary \ref{nss}, the curve $\overline{C}_{P,n}$ cannot be 
supersingular. 

If $X$ is supersingular, the Frobenius slope of $X$ is $1/2$. 
Thus the latter claim follows from Corollary \ref{cla}(3). 
\end{proof}

\begin{corollary}%\label{mono}
Assume that $q=p$ and $p$ is odd. 
Let $\xi \in W_{n,q,\rm prim}^{\vee}$ and $\psi \in \F_p^{\vee} \setminus \{1\}$. 
Let $\pi \colon \mathbb{A}^2_{\F_q} \to \mathbb{A}^1_{\F_q},\ (x,t) \mapsto x$ 
and $\pi' \colon \mathbb{A}^2_{\F_q} \to \mathbb{A}^1_{\F_q},\ (x,t) \mapsto t$. 
Let 
\[
\mathscr{F}_{\xi,\psi}
:=R^1\pi'_!(\pi^\ast\mathscr{L}_{\xi}(x,0'\mathrm{s}) \otimes 
\mathscr{L}_{\psi}(tx)) \bigl(\tfrac12\bigr). 
\]  
Then $\mathscr{F}_{\xi,\psi}$ does not have finite monodromy. 
\end{corollary}

\begin{proof}
%Take a faithful character ${}^0 \xi \in W_{n,p}^{\vee}$. Define ${}^0 \psi \in \F_p^{\vee}$ by ${}^0 \psi(x)={}^0 \xi(0'\mathrm{s},x)$ for $x \in \F_p$. Let $d \in \F_q^{\times}$ satisfy $\psi(x)={}^0 \psi \circ \Tr_{q/3}(dx)$. Write $\xi=\xi_a$ with $a=(a_0,\ldots,a_{n-1})$ as in Lemma \ref{ch}, let $c:=a^{-1}b^{1/3}$, and $t:=-cd^{-1}$. 
%For a morphism $f \colon X \to \mathbb{A}^1$, let $\mathscr{L}_{{}^0 \psi}(f(x))$ denote the pull-back of  $\mathscr{L}_{{}^0 \psi}(x)$ by $f$. 
Let $t \in \F_q$. 
By the proper base change theorem and Lemma \ref{chd}, we have
\[
(\mathscr{F}_{\xi,\psi})_{\bar{t}}
\simeq H_{\rm c}^1(\mathbb{A}^1,\mathscr{L}_{\xi}(x,0'\mathrm{s}) \otimes 
\mathscr{L}_{\psi}(tx))\bigl(\tfrac12\bigr) \simeq H_{\rm c}^1(\mathbb{A}^1,\mathscr{L}_{\xi'}(x,0'\mathrm{s}))\bigl(\tfrac12\bigr) 
\]
with some $\xi' \in W_{n,q,\mathrm{prim}}^{\vee}$. 
By \cite[Theorem 1.3]{Li}, the Frobenius eigenvalues of 
$\mathrm{Frob}_q^\ast$ on $(\mathscr{F}_{\xi,\psi})_{\bar{t}}$ are not a root of unity. 
The sheaf $\mathscr{L}_{\xi}(x,0'\mathrm{s})$ is geometrically irreducible, its determinant is arithmetically of finite order and pure of weight zero.  
Hence the claim follows (cf.\ \cite[Theorem 8.14.4]{Ka}). 
\end{proof}

Assume that $n=2$ and $p=l=3$. We simply write $W$ for $W_{2,q}$ as before. 
For $\xi \in W_{\mathrm{prim}}^{\vee}$ and $\alpha \in K_{P,\psi}$, let 
\[
S_{\xi,\alpha}:=\sum_{x \in \F_{q^s}} \psi(\Tr_{q^s/q}(\alpha x)) \xi(\Tr_{q^s/q}(x,0)). 
\]
Finally, we state our main theorem of this paper. 
\begin{corollary}\label{c43}
Assume $n=2$ and $p=l=3$. 
\begin{itemize}
\item[{\rm (1)}]
We have 
\[
L_{\overline{C}_{P,2}/\F_{q^s}}(T)=\prod_{\xi \in W_{\rm prim}^{\vee}} \prod_{a \in K_{P,\psi}}
(1+S_{\xi,\alpha} T+q^s T^2). 
\]
\item[{\rm (2)}] 
The Frobenius slopes of $\overline{C}_{P,2}$ with respect to $\mathrm{Fr}_q^\ast$ are 
$\{1/3,\, 2/3\}$. 
\end{itemize}
\end{corollary}
\begin{proof}
We show (1). 
By $n=2$, we have 
\[
L_{\overline{C}_{P,2}/\F_{q^s}}(T)=L_{V_2/\F_{q^s}}(T). 
\]
By Corollary \ref{chd}, we have 
\begin{align*}
\mathscr{L}_{\psi}(\alpha x) \otimes \mathscr{L}_{\xi}(x,0)
 \simeq \mathscr{L}_{\xi_{a'}}(x,0), 
\end{align*}
Thus, by Corollary \ref{cc2}, 
\begin{align*}
\det(1-\mathrm{Fr}_{q^s}^\ast T; H_{\rm c}^1(\mathbb{A}^1,\mathscr{L}_{\psi}(\alpha x) \otimes \mathscr{L}_{\xi}(x,0)))
&=\det(1-\mathrm{Fr}_{q^s}^\ast T; H_{\rm c}^1(\mathbb{A}^1,\mathscr{L}_{\xi_{a'}}(x,0))) \\
&=1+S_{\xi_{a'}} T+q^s T^2. 
\end{align*}
Hence the claim follows from Corollary \ref{cla}
and $S_{\xi,\alpha}=S_{\xi_{a'}}$. 
The claim (2) follows from Corollary \ref{sl}. 
\end{proof}
\subsection*{Acknowledgement} 

T.\ I.\ is supported by JSPS KAKENHI Grant Numbers 23K20786, 24K21512 and 25K00905.
D.\ T.\ is supported by JSPS KAKENHI Grant Number 25KJ0122. 
T.\ T.\ is supported by JSPS KAKENHI Grant Numbers 25K06959
and 23K20786.

\appendix
\section{Determinant}

In this appendix, we show that the Frobenius determinant on the first \'etale cohomology group with suitable coefficients is equal to a Gauss sum multiplied by a root of unity. This provides a generalization of Lemma~\ref{laumon}.

Let $q$ be a power of an odd prime $p$. 
The following is shown similarly to \cite[the proof of Theorem~7.1]{Ka0a}.
\begin{lemma}\label{katz}
Let $\mathscr{L}$ be a smooth 
$\overline{\mathbb{Q}}_{\ell}$-sheaf pure of weight $0$ on 
a smooth curve $X^0$ over $\F_q$. 
Let $X:=X^0_{\F}$. 
Assume that 
\[
\det(-\mathrm{Fr}_q^\ast; H_{\mathrm{c}}^1(X,\mathscr{L}))
\in \mathbb{Z}[\zeta_{p^n}] \quad \textrm{with an integer $n \ge 0$}.
\] 
Furthermore, we suppose that the natural map 
\[
H_{\rm c}^1(X,\mathscr{L})
\to H^1(X,\mathscr{L})
\]
is an isomorphism. 
We take $\psi\in\mathbb{F}_q^\vee\setminus\{1\}$ and 
denote by 
\[
G(\psi)=\sum_{x\in\mathbb{F}_q}\psi(x^2)
\]
the quadratic Gauss sum. 
If we set $s:=\dim H_{\mathrm{c}}^1(X,\mathscr{L})$, we have 
\[
\frac{\det(-\mathrm{Fr}_q^\ast;  
H_{\mathrm{c}}^1(X,\mathscr{L}))}
{G(\psi)^s}
\]
is a root of unity.
\end{lemma}

\begin{proof}
Put 
\[
\alpha := \det(-\mathrm{Fr}_q^\ast; H_{\mathrm{c}}^1(X,\mathscr{L}))
\in \mathbb{Z}[\zeta_{p^n}].
\]
Since $\mathscr{L}$ is pure of weight $0$, 
Weil~II \cite{Del2} implies that 
$H_{\rm c}^1(X, \mathscr{L})$
and $H_{\rm c}^1(X,\mathscr{L}^{\vee})$ are mixed of weights $\leq 1$. 
By Poincar\'e duality and the assumption, 
the natural map
\[
H^1_{\rm c}(X,\mathscr{L}^{\vee}) \to 
H^1(X,\mathscr{L}^{\vee})
\]
is an isomorphism and 
\[
H_{\rm c}^1(X,\mathscr{L})
\simeq H^1(X,\mathscr{L}^{\vee})^{\vee}(-1)
\simeq H^1_{\rm c}(X,\mathscr{L}^{\vee})^{\vee}(-1). 
\]
Hence 
$H_{\rm c}^1(X,\mathscr{L})$ is mixed of weights $\ge 1$ 
and therefore pure of weight $1$. 
Thus for any embedding 
$\iota \colon \overline{\mathbb{Q}}_{\ell}\hookrightarrow\mathbb{C}$,
the complex absolute value satisfies
\[
|\iota(\alpha)| = q^{s/2}.
\]
Similarly, for the Gauss sum, one has 
$G(\psi)\in\mathbb{Z}[\zeta_{p^n}]$ and 
$|\iota(G(\psi))| = q^{1/2}$ for any $\iota$. 
Hence
\begin{equation}\label{abv}
\biggl|\iota\!\left( \frac{\alpha}{G(\psi)^s} \right)\biggr|
 = 1.
\end{equation}

Let $\mathfrak{p}'$ be a finite place of 
$\mathbb{Q}(\zeta_{p^n})$ with $\mathfrak{p}'\nmid p$. 
Then both $\alpha$ and $G(\psi)$ in 
$\mathbb{Z}[\zeta_{p^n}]$
are $\mathfrak{p}'$-units. 

Since all archimedean absolute values of $\alpha/G(\psi)^s$
are $1$ by \eqref{abv}, the product formula implies 
that the $p$-adic valuation is zero. 
Therefore
\[
\frac{\alpha}{G(\psi)^s} \in 
\mathbb{Z}[\zeta_{p^n}],
\]
i.e.\ it is an algebraic integer, and by \eqref{abv}
all of its complex absolute values are equal to $1$.

By Kronecker's theorem, an algebraic integer whose
Galois conjugates all have absolute value $1$ 
is a root of unity. 
Thus $\alpha/G(\psi)^s$ is a root of unity.
\end{proof}
\begin{corollary}\label{kkatz}
Let $\mathscr{K}_{\chi}(x)$ be the Kummer 
sheaf on $\mathbb{G}_{\mathrm{m},\, \F_q}$ defined by 
$y^2=x$ and a character $\chi \in \F_2^{\vee}$. 
Let $n$ be a positive integer and let 
$W_{n,q} := \mathbb{W}_n(\F_q)$, 
where $\mathbb{W}_n$ denotes the $p$-typical Witt vector scheme of length $n$.  
For $\chi \in \F_2^{\vee}$ and a nontrivial character 
 $\xi \in W_{n,q}^{\vee} \setminus \{1\}$, 
 we define 
\[
V_{\chi, \xi}:=
\begin{cases}
H_{\rm c}^1(\mathbb{A}^1,\mathscr{L}_{\xi}(x,0's))
&\textrm{if $\chi = 1$}, \\
H_{\rm c}^1(\mathbb{G}_{\mathrm{m}}, \mathscr{K}_{\chi}(x) \otimes 
\mathscr{L}_{\xi}(x,0's)) & \textrm{if $\chi \neq 1$}. 
\end{cases}
\]
Then 
\[
\frac{\det(-\mathrm{Fr}_q^\ast; V_{\chi,\xi})}{G(\psi)^{\dim V_{\chi,\xi}}} 
\]
is a root of unity.

\end{corollary}
\begin{proof}
We simply write $\mathscr{L}$ for $\mathscr{K}_{\chi}(x) \otimes \mathscr{L}_{\xi}(x,0's)$.
We define $X^0$ to be $\mathbb{A}^1_{\F_q}$
if $\chi=1$ and $\mathbb{G}_{\mathrm{m}, \, \F_q}$ if $\chi \neq 1$. 
We apply Lemma \ref{katz}. 
By the trace formula, 
\[
S_m:=-\Tr(\mathrm{Fr}_{q^m}^\ast; H_{\rm c}^1(X,\mathscr{L})) \in \mathbb{Z}[\zeta_{p^n}]
\quad \textrm{for an integer $m \ge 1$}. 
\]
We define the $L$-function by
\[
L_{\mathscr{L}}(T):=\exp\left(\sum_{m=1}^{\infty} \frac{S_m}{m} T^m\right)
\in \mathbb{Q}(\zeta_{p^n})[[T]]. 
\]
We have $H_{\rm c}^i(X,\mathscr{L})=0$ for $i \neq 1$. 
Hence the function $L_{\mathscr{L}}(T)$ equals 
$\det(1-\mathrm{Fr}_q^\ast T; H_{\rm c}^1(X,\mathscr{L}))$ and the leading 
coefficient is 
\[
\alpha:=\det(-\mathrm{Fr}_q^\ast; H_{\rm c}^1(X, \mathscr{L})) \in \mathbb{Q}(\zeta_{p^n}). 
\]
Since $V_{\chi,\xi}\subset H^1(Y)$ for a smooth projective curve $Y$,
the eigenvalues of $\Fr_q^\ast$ on $V_{\chi,\xi}$ are algebraic integers by 
\cite{Del2}.  Hence $\alpha$ is an algebraic integer and $\alpha \in \mathbb{Z}[\zeta_{p^n}]$. 

Since $\mathscr{L}$ is 
ramified at every point of 
$\overline{X} \setminus X$ and has rank one, 
the natural map 
\[
H_{\rm c}^1(X,\mathscr{L})
\to H^1(X,\mathscr{L})
\]
is an isomorphism. 
The latter assertion follows from Lemma \ref{katz}. 
\end{proof}

\end{document}